\documentclass[12pt]{amsart}
\usepackage[english]{babel}
\usepackage{amsmath,amssymb, amsthm,amsbsy} 
\usepackage{graphicx}
\usepackage{geometry}
\usepackage{hyperref}  
\hypersetup{colorlinks=true, linkcolor=blue, anchorcolor=blue, 
citecolor=red, filecolor=blue, menucolor=blue,
urlcolor=blue}

\numberwithin{equation}{section}
\newtheorem{theorem}{Theorem}[section]
\newtheorem{corollary}[theorem]{Corollary}
\newtheorem{proposition}[theorem]{Proposition}
\newtheorem{lemma}[theorem]{Lemma}
\theoremstyle{remark}
\newtheorem{remark}{Remark}[section]

\theoremstyle{definition}

\newcommand{\R}{\mathbb{R}}
\newcommand{\C}{\mathbb{C}}
\newcommand{\N}{\mathbb{N}}
\newcommand{\Z}{\mathbb{Z}}

\newcommand{\lap}{\bigtriangleup}

\newcommand{\an}[1]{\langle #1 \rangle}

\newcommand{\grad}{\bigtriangledown}

\def\XXint#1#2#3{{\setbox0=\hbox{$#1{#2#3}{\int}$ }
\vcenter{\hbox{$#2#3$ }}\kern-.58\wd0}}

\begin{document}

\title
[Strichartz on curved space]
{Strichartz estimates for the Dirac equation on spherically symmetric spaces}

\begin{abstract}
\end{abstract}



\author{Federico Cacciafesta}
\address{Federico Cacciafesta: Dipartimento di Matematica, Universit$\grave{\text{a}}$ degli studi di Padova, Via Trieste, 63, 35131 Padova PD, Italy.}
\email{cacciafe@math.unipd.it}

\author{Anne-Sophie de Suzzoni}
\address{Anne-Sophie de Suzzoni: CMLS, \'Ecole Polytechnique, CNRS, Universit\'e Paris-Saclay, 91128 PALAISEAU Cedex, France.}
\email{anne-sophie.de-suzzoni@polytechnique.edu}

\subjclass[2010]{35J10, 35B99.}

\begin{abstract}
We prove local in time Strichartz estimates for the Dirac equation on spherically symmetric manifolds. As an application, we give a result of local well-posedness for some nonlinear models.
\end{abstract}

\maketitle


\section{Introduction}

In this paper we continue the study of the dynamics of the Dirac equation on curved spaces that we began in \cite{cacdes1}, in which we proved weak dispersive estimates for the flow in some different frameworks. 
%
We recall that the general form of the Dirac operator on a manifold $M$ with a given metric $g_{\mu\nu}$ is the following
 \begin{equation}\label{eq}
\mathcal{D}=i\gamma^a e^\mu_a D_\mu
\end{equation}
where the matrices $\gamma_0=\beta$ and $\gamma^j=\gamma^0\alpha_j$ for $j=1,2,3$ with
\begin{equation}\label{diracmatrices}
\alpha_k=\left(\begin{array}{cc}0 & \sigma_k \\\sigma_k & 0\end{array}\right),\quad k=1,2,3,\qquad\beta=
\left(\begin{array}{cc}I_2 & 0 \\0 & -I_2\end{array}\right)
\end{equation}
and
 \begin{equation}
\sigma_1=\left(\begin{array}{cc}0 & 1 \\1 & 0\end{array}\right),\quad
\sigma_2=\left(\begin{array}{cc}0 &
-i \\i & 0\end{array}\right),\quad
\sigma_3=\left(\begin{array}{cc}1 & 0\\0 & -1\end{array}\right),
\end{equation}
 $e^\mu$  is a \emph{vierbein} (i.e. a set of matrices that, essentially, connect the curved space-time to the Minkowski one and corresponds to a choice of frame for the tangent space in Cartesian formalism) and $D_\mu$ defines the covariant derivative for fermionic fields. For all the details on the construction and properties of the Dirac operator on a non-flat background we refer to \cite{cacdes1,parktoms}.
In what follows, we shall again restrict to metrics $g_{\mu\nu}$ having the following structure
\begin{equation}\label{struct}
g_{\mu\nu} = \left \lbrace{\begin{array}{ll}
\phi(t) & \textrm{ if } \mu= \nu = 0\\
0 & \textrm{ if } \mu\nu = 0 \textrm{ and } \mu\neq \nu\\
-h_{\mu\nu}(\overrightarrow x) & \textrm{ otherwise ,}\end{array}} \right.
\end{equation}
that is, decouple time and space. Also, as a further simplification, we assume $\phi(t)=1$; this after a change of variable in time, actually allows to cover all the possible choices of $\phi(t)$ strictly positive for all $t$.  In this setting, the (Cauchy problem for the) Dirac equation assumes the convenient form
\begin{equation}\label{diraceqcurv}
\begin{cases}
i \partial_t u - H u = 0,\\
u(0,x)=u_0(x)
\end{cases}
\end{equation}
where $H$ is an operator such that $H^2=-\Delta_h+\frac14\mathcal{R}_h+m^2$, $\Delta_h$  is the Laplace-Beltrami operator for Dirac spinors, that is, $\lap_h = D^jD_j$ where $D_j$ is the covariant derivative for Dirac spinors that we properly define later, and $\mathcal R_h$ is the scalar curvature associated to the spatial metrics $h$.  As a consequence, it can be proved that if $u$ solves equation \eqref{diraceqcurv} then $u$ also solves the equation
\begin{equation}\label{quaddir}
-\partial_t^2 u  + \lap_h u - \frac14 \mathcal R_h u -m^2 u = 0.
\end{equation}
We point out that the scalar curvature term vanishes when specializing formula above to the standard Minkowski space, so that in this case this formula recovers the well-known one. What is more, the covariant derivatives in the usual choice of vierbein is simply given by $D_\mu = \partial_\mu$.
This remark is extremely useful, as it often allows to translate some well-known facts for the wave or Klein-Gordon equations to the Dirac setting. In particular, this remark is key for proving dispersive estimates for the Dirac flow, also in presence of small potential perturbations: in \cite{cacdes1} indeed, the classical Morawetz-multiplier method was adapted to equation \eqref{quaddir} to obtain local smoothing estimates for different choices of the metrics $h_{\mu\nu}$, which include asymptotically flat and some warped product manifolds.

The subsequent natural step would now be to prove {\em Strichartz estimates} for the Dirac equation in these settings; unfortunately, it is not really possible to apply the standard Duhamel trick combined with local smoothing (see e.g. \cite{boussdanfan}) to deal with equation \eqref{diraceqcurv} as a perturbation of of the flat Dirac equation, due to the fact that we are in presence of a high order perturbation. Therefore, as it is often the case when it comes to variable coefficients dispersive PDEs, some different strategies need to be developed in order to obtain Strichartz estimates

%
\medskip
In this manuscript we focus on the case of {\em spherically symmetric manifolds}, i.e. manifolds $(M,g)$ defined by $M=\mathbb{R}_t\times \Sigma$ where $\Sigma=\mathbb{R}_r^{+}\times\mathbb{S}^2_{\theta,\phi}$ equipped with the Riemannian metrics
\begin{equation}\label{sphermet}
d\sigma=dr^2+\varphi(r)^2d\omega_{\mathbb{S}^2}^2
\end{equation}
where $d\omega_{\mathbb{S}^2}^2=(d\theta^2+\sin^2\theta d\phi^2)$ is the Euclidean metrics on the 2D sphere $\mathbb{S}^2$. Notice that taking $\varphi(r)=r$ reduces $\Sigma$ to the standard 3D euclidean space, and therefore $M$ to be the standard Minkowski space. Obviously, various different assumptions can be made on the functions $\varphi(r)$ leading to very different geometrical situations; we assume the following set of hypothesis.
\medskip

\textbf{Assumptions (A1)} Take $\varphi(r)\in C^\infty(\mathbb{R^+})$ strictly positive on $(0,+\infty)$, such that
\begin{equation}\label{assh}
\varphi(0)=\varphi^{(2n)}(0)=0,\qquad \varphi'(0)=1,\qquad \Big|\frac{\varphi'(r)}{\varphi(r)}\Big|\leq C\ \qquad \inf_{r\geq 1} \varphi(r) > 0. 
\end{equation}
We also assume that the scalar curvature of $M$ is bounded. It might be negative though.

\begin{remark}
These assumptions are fairly natural for the present contest: indeed, in order for a a smooth and spherically symmetric manifold $M$ to have a global metrics of the form \eqref{sphermet} the function $\varphi(r)$ must be the restriction to $\mathbb{R}^+$ of a $C^\infty$ odd function with $\varphi(0)=0$ and $\varphi'(0)=1$.  In fact, these are essentially the same assumptions made by the authors in \cite{banduy} in order to obtain similar results for the Schr\"odinger equation. Also, these assumptions guarantee that the manifold $M$ is smooth (see \cite{pet} paragraph 1.3.4).
 \end{remark}
 
The basic advantage in having a spherically symmetric manifold relies on the fact that it is possible to decompose the Dirac operator (and analogously its flow) in a sum of "radial" Dirac operators (see section \ref{radsec} for details) and so, somehow, handle the geometric term after a change of variable as a potential perturbation. This strategy is strongly inspired by \cite{banduy}, in which the authors obtain local and global in time Strichartz estimates for the Schr\"odinger flow on spherically symmetric manifolds. However, we should stress the two main differences with respect to their case, which also represent the two main difficulties here: first, the "radial" decomposition for the Dirac operator is much more subtle, and forces to work on 2-dimensional angular spaces, due to the fact that the Dirac operator does {\em not} preserve radial spinors. Second, this approach naturally produces, as we will see, {\em scaling critical} potential perturbations, and while for the Schr\"odinger equation with inverse square potential dispersive estimates are well known (see \cite{burq1}, \cite{burq2}), for the Dirac equation with a Coulomb-type potential only some weak local smoothing effect has been proved (see \cite{cacser}) but nothing is known at the level of Strichartz estimates.
\medskip

Before stating our main results, let us fix some notations.
\medskip

{\bf Notations.} We will use the standard notation $L^p$, $\dot{H}^s$, $H^s$, $W^{p,q}$ to denote, respectively, the Lebesgue and the homogeneous/non homogeneous Sobolev spaces of functions from $\R^3$ to $\C^4$. We will use the same notation to denote these functional spaces on the (spatial) manifold $(M_h,h)$, which is in our structure \eqref{struct}, i.e. with time and space already decoupled, by adding the dependence $L^p(M)$, $\dot{H}^s(M)$, $H^s(M)$, $W^{p,q}(M)$: e.g., the norm $L^p(M)$ will be given by
$$
\|f\|_{L^p(M)}^p:=\int |f(x)|^p \sqrt{\det(h(x))}dx
$$
and so on. In particular, due to the spherically symmetric structure of the metrics \eqref{sphermet}, for a radial function $f_{rad}(|x|)$ we will have
$$
\|f_{rad}\|_{L^p(M)}^p:=\int_0^{+\infty} |f_{rad}(r)|^p \varphi(r)^2dr 
$$
and similarly for the Sobolev spaces. Note that since we are dealing with vectors in $\C^4$, $|f(x)|$ should be understood as 
$$
|f(x)| = \sqrt{\an{f(x),f(x)}_{\C^4}}
$$
and because we are dealing with spinors, the derivatives we take for the Sobolev norms are covariant derivatives. For instance, the $\dot H^1(M)$ norm of some map $f$ is given by
$$
\|f\|_{\dot H^1(M)}^2 =  \sum_{j=1}^{3} \int  |D_j f(x)|^2 \sqrt{\textrm{det }h(x)} dx,
$$
which can be written, since $M$ is smooth as
$$
\|f\|_{\dot H^1(M)}^2 =  - \int  \an{\lap_h f(x),f(x)}_{\C^4} \sqrt{\textrm{det }h(x)} dx
$$
with $\lap_h = D^jD_j$ the covariant Laplace-Beltrami operator for Dirac (bi)spinors.

The norms in time will be denoted by $L^p_t(I)$, and the time interval $I$ will be allowed to be bounded or unbounded. The mixed Strichartz spaces will be denoted by $ L^p_t(I)L^q(M)=L^p(I;L^q(M,\C^4))$. In what follows we will also need Lebesgue spaces which separate radial from angular regularity: we will use the notation
$$
\|f\|_{L^p_{\varphi(r)^2dr}L^q_\omega}:=\int_0^{+\infty}\|f(r,\cdot)\|_{L^q(\mathbb{S}^2)}^q\varphi(r)^2 dr.
$$
The operator $\Lambda_\omega$ will denote the fractional Laplace-Beltrami operator on $\mathbb{S}^2$, that is $\Lambda^s_\omega=(1-\Delta_{\mathbb{S}^2})^{s/2}$. A crucial role in our analysis will be played by the so called partial wave decomposition, a detailed discussion of which is postponed to section 3. In order to state the result, we only limit here to recall the notations from \cite{thaller}: there exists an orthogonal decomposition
$$
L^2(\mathbb{S}^2)^4\cong\bigoplus_{j,m_j,k_j} \mathcal{H}_{j,m_j,k_j}
$$
where the spaces $\mathcal{H}_{j,m_j,k_j}$ are called partial wave subspaces. They are defined as in \cite{thaller} Subsection 4.6.4, each $\mathcal H_{j,m_j,k_j}$ is of dimension $2$ and the indexation works as $j\in \frac12 + \N$, $m_j \in (\frac12 + \Z) \cap [-j,j]$ and $k_j = \pm (j+\frac12)$.

In order to have a more compact notation, we introduce the spaces 
\begin{equation}\label{Hj}
\mathcal H_j = \bigoplus_{m_j,k_j} \mathcal H_{j, m_j,k_j}
\end{equation}
and with the convention $\mathcal H_{-1/2} = \{0\}$, for $n\in \N$,
\begin{equation}\label{Pn}
\mathcal P_n = \mathcal H_{n-1/2} \oplus \mathcal H_{n+1/2}.
\end{equation}
We also write $\mathcal S_n$ the spherical harmonics from $S^2$ to $\C^4$ with degree $n$. Note that
$$
\mathcal S_n \subseteq \mathcal P_n.
$$

\medskip
We are now ready to state our first Theorem, that 	contains local-in-time Strichartz estimates for the Dirac flow under the Assumptions {\bf (A1)} for initial condition with prescribed angular component. 

\begin{theorem}\label{teo1}
Let $g$ be as in \eqref{struct}, $h$ having the structure \eqref{sphermet} and satisfying Assumptions {\bf (A1)}. Then for any bounded interval $I=(0,T)$, $T>0$, there exists a constant $C_T$ such that the solutions $u$ to \eqref{diraceqcurv} with initial condition $u_0\in H^s((0,+\infty)\varphi(r)^2dr)\otimes\mathcal{P}_{n}$  for a fixed $n$ satisfy estimates
\begin{equation}\label{strichartznew1}
\left\|u\left(\frac{\varphi(r)}r\right)^{1-\frac2q}\right\|_{L^p_t(I)L^q(M)}\leq C_T\langle n\rangle \|u_0\|_{H^s(M)}
\end{equation}
provided $s=2/p$, $\frac2{p} + \frac2{q} = 1$ and $p \in ]2,\infty]$ but also 
\begin{equation}\label{strichartznew2}
\left\|u\left(\frac{\varphi(r)}r\right)^{1-\frac2q}\right\|_{L^p_t(I)L^q(M)}\leq C_T\langle n \rangle \|u_0\|_{H^s(M)}
\end{equation}
provided that $m\neq 0$, $s=1/p$, $\frac2{p} + \frac3{q}= \frac32$ and $p\in [2,\infty]$.
\end{theorem}

\begin{remark}
Notice that if $\varphi(r)\geq r$, i.e. if the volume element of $M$ grows faster than the one in the Euclidean case, these estimates actually produce a gain in space with respect to the Strichartz estimates in the flat case. In fact, the growth of the term $\varphi(r)\geq r$ can be related to the sign of the curvature of the manifold, and more precisely to the one of the {\em tangential sectional curvature}. Indeed, the tangential component of the sectional curvature $sec_{tan}$ in the setting of spherically symmetric manifolds is given by
$$
sec_{tan}=-\frac{(\varphi')^2-1}{\varphi^2}.
$$
Suppose that this is non-positive for any $r\geq r_0$, for some $r_0$. Then $\varphi'(r)\geq1$ for all $r\geq r_0$ because $\varphi'(r)$ is continuous and $\varphi$ is assumed positive: this rules out the case $\varphi'(r)\leq -1$ for $r\geq r_0$. But then
$$
\varphi(r)-\varphi(r_0)=\int_{r_0}^r\varphi'(s)ds\geq r-r_0
$$
which for $r\rightarrow+\infty$ gives
\begin{equation}\label{ccp}
\frac{\varphi(r)}{r}\geq1.
\end{equation}
Therefore, the negativity of the tangential sectional curvature results in \eqref{ccp}, which ensures in fact a "gain" in estimates \eqref{strichartznew1}-\eqref{strichartznew2}. This fact was already remarked in \cite{banduy}.
\end{remark}

\begin{remark}
As it is often the case when dealing with potential perturbations, the estimates in the massless case need the full non homogeneous Sobolev norm on the initial condition: the $L^2$ norm is indeed needed to control derivatives in the weighted case (see forthcoming Lemma \ref{lemma1}).
\end{remark}

\begin{remark}
As we mentioned, the basic idea of the proof relies, roughly speaking, in using partial wave decomposition to reduce the problem to a radial one, and then introducing suitable weighted spinors so that, in the new variable, the equation becomes a "flat" Dirac equation with a perturbative term that can be seen (and handled) as a potential perturbation. Then, the validity of weighted Strichartz estimates for the dynamics on this curved setting corresponds to the validity of Strichartz estimates for some potential perturbed flow on $\mathbb{R}^3$ (see Proposition \ref{boundpot}).
\end{remark}

\begin{remark}
It is interesting to compare the present situation with the Schr\"odinger equation counterpart  (see in particular Remark 2.7 in \cite{banduy}). The Laplace operator associated to a metric with the form \eqref{sphermet} is indeed, in a generic dimension $n\geq3$, 
$$
\Delta_M=\partial^2_r+\frac{n-1}2\frac{\varphi'(r)}{\varphi(r)}\partial_r+\frac1{\varphi^2(r)}\Delta_{S^{n-1}}.
$$
Therefore, when taking a polynomial-type $\varphi(r)=r^m$, for some integer $m\geq1$, the radial part of the operator above reduces to the radial part of the Laplacian on $\mathbb{R}^N$ with $N=1+m(n-1)\geq n$. If one considers radial solutions then, it is possible to introduce a weighted function such that the restricted radial equation becomes indeed a radial, flat Schr\"odinger equation with a potential which, in general, will be bounded and will have again a scaling critical behaviour at infinity.
\end{remark}
\medskip

By relying on orthogonality and unitarity of spherical harmonics, we can deduce from Theorem \ref{teo1} weighted Strichartz estimates with loss of angular derivatives for generic initial condition. 

We introduce the spaces $H^{a,b}$ for $a,b \in\R$ by defining the norms
$$
\|f\|_{H^{a,b}} = \Big[ \sum_{j,m_j,k_j,\pm}\Big(\an{k_j}^{2b}  \|f_{j,m_j,k_k,\pm}\|_{L^2(\varphi^2(r)dr)}^2 + \|f_{j,m_j,k_k,\pm}\|_{H^{a}(\varphi^2(r)dr)}^2\Big) \Big]^{1/2}
$$
where 
$$
f_{j,m_j,k_j,\pm }  = \an{f,\Phi_{m_j,k_j}^\pm}_{L^2(\mathbb{S}^2)}
$$
with $\{\Phi_{m_j,k_j}^+ , \Phi_{m_j,k_j}^-\}$ an orthonormal basis of $\mathcal H_{j,m_j,k_j}$. In other words, we are taking $a$ derivatives in radial coordinates, $b$ derivatives in angular coordinates, and the $L^2$ norm on the whole manifold.

With these notations, interpolating the estimates in Theorem \ref{teo1} with the conservation of the $L^2$ norm and using Littlewood-Paley theory on the sphere, we get the following.

\begin{corollary}\label{cor1}
Let $g$ be as in \eqref{struct}, $h$ having the structure \eqref{sphermet} and satisfying Assumptions {\bf (A1)}. Let $p,q\in [2,\infty]$ and $a,b \geq 0$. Assume either $p>2$, $b>\frac4{p}$, $\frac1{p} + \frac1{q} = \frac12$ and $\frac2{pa} + \frac{2}{pb} < 1$ or $m\neq 0$, $b\geq \frac{3}{p}$, $\frac2{p} + \frac3{q} = \frac32$ and $\frac1{pa} + \frac2{pb} \leq 1$. Then for any bounded interval $I=(0,T)$, $T>0$, there exists a constant $C_T$ such that the solutions $u$ to \eqref{diraceqcurv} with initial condition $u_0$ such that $ u_0\in H^{a,b}$ satisfy the estimates
\begin{equation}\label{strichartznew1bis}
\left\|u\left(\frac{\varphi(r)}r\right)^{1-\frac2q}\right\|_{L^p_t(I,L^q)}\leq C_T\| u_0\|_{H^{a,b}}.
\end{equation}
\end{corollary}

\begin{remark} It is reasonable to expect that these results can be adapted to the more general setting of warped product manifolds, i.e. manifolds $(M,g)$ with $M=\mathbb{R}_t\times \Sigma$ and $\Sigma$ equipped with the Riemannian metrics
\begin{equation}
d\sigma=dr^2+\varphi(r)^2d\beta
\end{equation}
with some compact geometry $\beta$ which is now not necessarily the sphere. One should get indeed
$$
\mathcal D_h = \gamma^1 \partial_r + \frac1{\varphi(r)} \mathcal D_\beta
$$
with $\mathcal D_\beta $ not depending on $r$. If $\mathcal D_\beta$ is diagonalisable then one can produce the same type of theorems, provided $\beta$ induces a Littlewood-Paley theory. We intend to deal with this problem in forthcoming works. \end{remark}

\begin{remark}
Estimates involving angular regularity have been already widely investigated and exploited in the contest of the flat space. In particular, we mention \cite{machnach} in which is proved a the 3D endpoint Strichartz estimates with an $\varepsilon$-loss of angular regularity for the Dirac and wave equations, and \cite{cac2} (which is actually closer in spirit to the strategy of the present paper), in which the same problem is dealt with, also with the additional presence of small potentials.
\end{remark}

As an application of our estimates, we can prove a local well-posedness result in a subcritical regime for some nonlinear Dirac equations with "radial" initial conditions. In particular, we are interested in the study of the nonlinear Dirac equation in the form
\begin{equation*}
i\gamma^\mu D_\mu u-m u=\lambda \langle \gamma^0 u , u\rangle_{\C^4} u
\end{equation*}
which, after multiplying times $\gamma^0$ can be written in the equivalent way
\begin{equation}\label{nld0}
i\partial_t u+\mathcal{D} u-m \gamma^0u=\lambda \langle \gamma^0u, u\rangle \gamma^0 u
\end{equation}
to isolate the time.
The problem of studying local/global well posedness for equation \eqref{nld} (and, more in general, with polynomial type nonlinearities in the form $|\langle \gamma^0u, u\rangle|^{\frac{p-1}2}  u$, $p\geq 3$) in the flat setting has been addressed by several authors (see \cite{reed, dias,naj,escveg,machnach}). In particular, in \cite{escveg} the authors provided a fairly complete picture of well posedness in the subcritical range. Improvements involving additional angular regularity, exploiting some refined Strichartz estimates, were subsequently given in \cite{machnach}. We should also mention \cite{bejher}, in which the authors proved global well posedness (and scattering) for the cubic nonlinear Dirac equation with data in $H^1$, i.e. in the critical case. 

Here, relying on our new Strichartz estimates, we can prove the following result of local well posedness on a non flat background.

\begin{theorem}\label{teo3} Assume $\inf \frac{\varphi(r)}{r}> 0$ and Assumption A. Let $r >  0$, $r' = \max(r,2)$ and let $s_1 = \frac32 - \frac3{r'}$. Let $a ,b >s_1$ such that $a<2$ and
$$ 
\frac2{r'} \Big( \frac1{a-s_1} + \frac1{b-s_1}\Big) < 1 \textrm{ and }b>\frac1{r'} + \frac32 .
$$
Then, for all $R \geq 0$ there exists $T(R) > 0$ such that for all $u_0 \in H^{a,b}$ with $\|u_0\|_{H^{a,b}} \leq R$, the Cauchy problem
\begin{equation}\label{nld}
\begin{cases}
i \partial_t u - H u =| \langle \beta u,u\rangle|^{\frac{r}2}u,\\
u(0,x)=u_0(x) \in H^{a,b}
\end{cases}
\end{equation}
has a unique solution in $\mathcal C ([-T,T],H^{a,b})$ and the flow hence defined is continuous in the initial datum.
\end{theorem}

\begin{remark} Note that the condition on $a$ means that $a$ must be taken strictly bigger than the saling-critical regularity  $s_c = \frac32 - \frac1{r'} >s_1$ of the equation in the flat case.
\end{remark}

\begin{remark} Taking the initial datum $u_0$ in $H^a((0,+\infty)\varphi(r)^2dr)\otimes\mathcal{H}_j$, we have that $u_0$ belongs to $H^{a,b}$ for any $b$, and thus the equation admits a unique local solution in $\mathcal C([-T,T],H^{a,b})$ for any $b$ (though the time of existence depends on $b$) and the flow is continuous in the initial datum on $H^a((0,+\infty)\varphi(r)^2dr)\otimes\mathcal{H}_j$.
\end{remark}

We also have the following theorem for "radial" data.

\begin{theorem}\label{corsolmod} Assume $\inf \frac{\varphi(r)}{r} > 0$ and Assumption A. 
 Let $r>0$, $r' = \max(2,r)$ and $s_c = \frac32 - \frac1{r'}$. Take $(m_{1/2},k_{1/2})$ in $\{(-1/2,-1), (-1/2,1), (1/2,-1), (1/2,1)\}$. Assume $a>s_c$ then for all $R \geq 0$ there exists $T(R) > 0$ such that for all $u_0\in H^a((0,+\infty)\varphi(r)^2dr)\otimes\mathcal{H}_{1/2,m_{1/2},k_{1/2}}$ with $\|u_0\|_{H^{a}} \leq R$, the Cauchy problem
\begin{equation}
\begin{cases}
i \partial_t u - H u =| \langle \beta u,u\rangle|^{\frac{r}2}u\\
u(0,x)=u_0(x) \in H^{a}\otimes\mathcal{H}_{1/2,m_{1/2},k_{1/2}}
\end{cases}
\end{equation}
has a unique solution in $\mathcal C ([-T,T],H^{a}\otimes\mathcal{H}_{1/2,m_{1/2},k_{1/2}})$ and the flow hence defined is continuous in the initial datum. \end{theorem}

\begin{remark}  In particular the {\em Soler model}, that corresponds to the choice $r=2$, is locally well-posed in $H^a$ for $a> s_c = 1$.  \end{remark}

\begin{remark}
We stress the fact that the idea of relying on partial wave subspaces to define a nonlinear Dirac equation (with somehow improved results of well posedness) is not new and has been already exploited to give some partial results in the 3D cubic (flat) case (see \cite{machnach}), also in presence of external potentials (see \cite{cac1,cac2}). Indeed, the main advantage is in that the nonlinear dynamics preserves the partial wave decomposition provided the initial condition has prescribed angular component, more precisely its angular part belongs to one of the four $\mathcal{H}_{1/2}$ spaces.
\end{remark}


\section{Preliminaries: the Dirac equation on $\mathbb{R}^3$}

In this section we recall some results on the dispersive dynamics of the Dirac equation with potentials in the Euclidean setting; in particular, we show that local in time Strichartz estimates hold true in presence of bounded perturbations.

\subsection{Strichartz estimates}
Strichartz estimates for the Dirac equation in the Euclidean setting, both in the massless and massive case, are well known and in view of \eqref{quaddir} can be easily deduced by the corresponding ones for the wave and Klein-Gordon equations. We recall indeed that the solutions to the 3-dimensional Dirac equation
\begin{equation}\label{dirfree}
\begin{cases}
\displaystyle
 i\partial_tu+\mathcal{D}u+m\beta u=0,\quad u(t,x):\mathbb{R}_t\times\mathbb{R}_x^n\rightarrow\mathbb{C}^{N}\\
u(0,x)=u_0(x)
\end{cases}
\end{equation}
satisfy the following families of 

{\bf Strichartz estimates (S)}
\begin{itemize}
\item {\em Case $m=0$}: 
$$\|e^{it\mathcal{D}}u_0\|_{L^p_t(I)L^q(\mathbb{R}^3)}\lesssim \||D|^{\frac2p}u_0\|_{L^2(\R^3)}$$
$$
\left\||D|^{-\frac2p}\int_0^te^{i(t-s)\mathcal{D}}F\right\|_{L^p_t(I)L^q(\mathbb{R}^3)}\lesssim \||D|^{\frac2{\tilde{p}}}F\|_{L^{\tilde{p}'}L^{\tilde{q}'}(\R^3)}
$$
provided both $(p,q)$ and $(\tilde{p},\tilde{q})$ satisfy the admissibility condition
\begin{equation}\label{wad}
\displaystyle \frac2p+\frac2q=1,\qquad2\leq p\leq\infty,\quad2\leq q<\infty
\end{equation}
holds.
\item
{\em Case $m\neq 0$}
$$\|e^{it(\mathcal{D}+\beta)}u_0\|_{L^p_t(I)L^q(\mathbb{R}^3)}\lesssim \|\langle D\rangle^{\frac1p}u_0\|_{L^2}
$$
$$
\left\|\langle D\rangle^{-\frac1p}\int_0^te^{i(t-s)\mathcal{D}}F\right\|_{L^p_t(I)L^q(\mathbb{R}^3)}\lesssim \|\langle D\rangle^{\frac1{\tilde{p}}}F\|_{L^{\tilde{p}'}L^{\tilde{q}'}(\R^3)}
$$
provided both $(p,q)$ and $(\tilde{p},\tilde{q})$ satisfy the admissibility condition
\begin{equation}\label{sad}
\displaystyle \frac2p+\frac3q=\frac32,\qquad2\leq p\leq\infty,\quad2\leq q\leq6
\end{equation}
holds.
\end{itemize}

Notice that the time interval $I$ can be bounded or unbounded. 

\begin{remark}
Using the fact that the wave flow commutes with Fourier multiplier, it is of course possible to move (some of) the derivatives on the initial data on the left hand side in estimates above. In particular, thanks to Sobolev inequalities, the estimate in the case $m\neq 0$ implies the estimate in the case $m=0$ and thus, we have the first one for any $m\in \R$.
\end{remark}

\begin{remark}
The problem of studying dispersive, and in particular global Strichartz, estimates for potential perturbations of the Dirac equation is quite well investigated. Indeed, it is now understood that, essentially, \emph{subcritical} (with respect to scaling) potentials do not provide any obstruction to dispersion: more precisely (see \cite{danfan}), the flow $e^{it(\mathcal{D}+V)}$ satisfies the same family of Strichartz estimates as in the free case as long as
$$
|V(x)|\leq \frac{\delta}{w_\kappa(x)}
$$
with $w_\kappa(x)=|x|(1+|\log|x||)^\kappa$, $\kappa>1$ and $\delta$ sufficiently small. Refined results can be obtained if one deals with radial potentials (see \cite{cacdan}), so that one can take into account angular regularity as well. When the perturbation becomes scaling critical, that is the case of the Coulomb potential $V(x)\cong 1/|x|$, the situation becomes considerably more complicated and it is not known whether the corresponding dynamics preserves Strichartz estimates or not: to the best of our knowledge, the only available result in this direction is \cite{cacser}, in which the authors were only able to prove a suitable family of  local smoothing estimates. It is worth noticing that this fact provides some major difference with respect to the Schr\"odinger (and wave) equations, for which scaling critical electric-potential perturbations, that are represented by inverse-square potentials, are known not to alter the dispersive dynamics. This difference might be understood by means of formula \eqref{quaddir}: this suggests indeed that the Dirac-Coulomb model should behave much closer to a system of wave equations with a scaling-critical first order perturbation, for which, as far as we know, no results are available.
\end{remark}

While proving global-in-time Strichartz estimates for a potential perturbation of the flow might require some technical tool, due essentially to the loss of derivatives in the free estimates that does not allow to directly rely on the $TT^*$ method, local-in-time estimates are much easier to obtain for "small" potentials. We prove the following

\begin{proposition}\label{boundpot}
Let $V$ be a continuous operator from $L^2(\R^3)$ to $L^2(\R^3)$ which is also continuous from $H^{1}$ to itself, $m\geq0$, $I$ a bounded time interval. Assume that the flow $S(t) = e^{it(\mathcal{D}_{\R^3}+m\beta+V)}$ is continuous from $H^{s}$ to $\mathcal C(\R,H^s)$ for $s\in [0,1]$ then for all $p,q,s \in [2,\infty]^2\times \R_+$ satisfying $\frac1{p} + \frac1{q} = 1$, $p>2$, $s=\frac2{p}$ or $m\neq 0$, $\frac2{p} + \frac3{q} = \frac32$, $s= \frac1{p}$, we have for all $u_0 \in H^{s}$,
$$
\|S(t) u_0\|_{L^p_t(I)L^q(\mathbb{R}^3)} \leq C(I) (1 + \|V\|_{H^{s} \rightarrow H^{s}} ) ( 1+ \|S(t)\|_{H^s \rightarrow L^\infty, H^s}) \|u_0\|_{H^s}.
$$ 
\end{proposition}

\begin{proof} In this proof, we write $\mathcal D$ for $\mathcal D_{\R^3}$.

We show the case when $m=0$, the other one being completely analogous.

We use Duhamel formula to represent the solution $u=e^{it(\mathcal{D}+m\beta+V)}u_0$ and take any admissible Strichartz norm. We have that
$$
u(t) = e^{it\mathcal D} u_0 -i \int_{0}^t e^{i(t-\tau) \mathcal D } (Vu)(\tau) d\tau.
$$
We use Strichartz estimates and Christ-Kiselev lemma to get
$$
\|u\|_{L^p_t(I)L^q(\mathbb{R}^3)} \lesssim \|u_0\|_{H^{2/p}} + \|Vu\|_{L^1(I)H^{2/p}(\R^3)}.
$$
We use that $V$ is continuous from $H^{2/p}$ to itself to get
$$
\|u\|_{L^p_t(I)L^q(\mathbb{R}^3)} \lesssim \|u_0\|_{H^{2/p}} + |I| \|V\|_{H^{2/p} \rightarrow H^{2/p}} \|u\|_{L^\infty(I)H^{2/p}(\R^3)}.
$$
We use the continuity of $S(t)$ from $H^{2/p}$ to $\mathcal C(\R,H^{2/p})$ to get
$$
\|u\|_{L^p_t(I)L^q(\mathbb{R}^3)} \lesssim \|u_0\|_{H^{2/p}} + |I| \|V\|_{H^{2/p} \rightarrow H^{2/p}} \|S(t)\|_{H^{2/p}\rightarrow L^\infty,H^{2/p}} \|u_0\|_{H^{2/p}}.
$$
This concludes the proof.

\end{proof}

\begin{remark}
Any $V$ which is the multiplication by a $W^{1,\infty}$ map is an admissible choice.
\end{remark}

\section{The radial Dirac Equation on symmetric manifolds}\label{radsec}

We devote this section to show how the Dirac equation writes in spherically symmetric manifolds, and how the introduction of weighted spinors transforms the equation into an equation with potential on Minkowski space. In this section, $\psi$ denotes a solution to the linear Dirac equation in a spherically symmetric manifold.

\subsection{The Dirac operator in spherical coordinates}

The construction of the Dirac operator on a 4D manifold is a delicate task, and requires the introduction of the so called \emph{vierbein} which, essentially, define some proper frames that connect the metrics of the manifold $(M,g)$ to the Monkowski one $\eta$; details can be found in the predecessor of this paper, \cite{cacdes1}, and in \cite{parktoms}. Anyway, when the metrics has the particular structure \eqref{sphermet} it is possible to write some explicit formulas by using spherical coordinates (similar calculations were developed in \cite{daude1}). First, let us recall that the Dirac equation can be written in the general form
\begin{equation}\label{gendir}
(\underline\gamma^jD_j-m)\psi=0
\end{equation}
where $m\geq0$ is the mass, the $\underline\gamma^j$ matrices are an adaptation of the standard one, i.e. they are a set of matrices satisfying the anticommuting relation
$$
\{\underline\gamma^i,\underline\gamma^j\}=2g^{ij},
$$
and can be written using the standard ones as $\underline\gamma^j=e^j_{\; a}\gamma^a$ where $e^j_{\; a}$ is a vierbein for $M,g$ and the $\gamma^a$ are the standard gamma matrices satisfying
\begin{equation}\label{flatanti}
\{\gamma^i,\gamma^j\}=2\eta^{ij}.
\end{equation}
Classicaly, one takes
\begin{equation}
\gamma^0=\left(\begin{array}{cc}\sigma_0 & 0 \\0 & -\sigma_0\end{array}\right),\qquad
\gamma^j=\left(\begin{array}{cc}0 & \sigma_j\\-\sigma_j & 0\end{array}\right),\quad
\end{equation}
where the $\sigma$ matrices are the well known Pauli matrices
\begin{equation}
\sigma_0=\left(\begin{array}{cc}1 & 0 \\0 & 1\end{array}\right),\quad
\sigma_1=\left(\begin{array}{cc}0 & 1 \\1 & 0\end{array}\right),\quad
\sigma_2=\left(\begin{array}{cc}0 &-i \\i & 0\end{array}\right),\quad
\sigma_3=\left(\begin{array}{cc}1 & 0\\0 & -1\end{array}\right).
\end{equation}
The differential operator $D_j$ is the covariant derivative for spinors, and it is defined as $D_j=\partial_j+\Gamma_j$ where $\Gamma_j$ is given by 
\begin{equation}\label{spinconn}
\Gamma_j=\frac18\omega_j^{\; ab}[\gamma_a,\gamma_b]
\end{equation}
which contains a purely algebraic part $[\gamma_a,\gamma_b]$ that corresponds to the generators of the underlying Lie algebra for Dirac bi-spinors, and a purely geometric one $\omega_j^{\; ab}$, namely the spin connection. It is given by
$$
\omega_j^{\; ab} = e_i^{\; a} \Gamma^i_{\; jk}  e^{kb} + e_i^{\; a} \partial_j e^{ib}
$$
and is characterized by the formula
\begin{equation}\label{carom}
de^{ a} + \omega^a_{\; b}\wedge e^{ b} = 0
\end{equation}
where $e^{ a} = e_j^{\; a} dx^j$ and $\omega^a_{\; b} = \omega_{j \; b}^{\; a} dx^j$. 

In our assumption on the metrics \eqref{sphermet}, by using spherical coordinates, it is natural to choose the dreibein (which connects the flat metrics to the spatial metrics $h$) :
$$
e_1=\partial_r,\qquad e_2=\frac1{\varphi(r)}\partial_\theta,\qquad e_3=\frac1{\varphi(r)\sin\theta} \partial_\phi,
$$
that is $e_j^{\; a}$ is the matrix
$$
\begin{pmatrix} 1 &0 &0 \\ 0 & \varphi(r) & 0 \\ 0&0 & \varphi(r) \sin \theta \end{pmatrix}.
$$
Thus, the associated dual 1-forms are
$$
e^1=dr,\qquad e^2=\varphi(r)d\theta,\qquad e^3=\varphi(r)\sin\theta d\phi
$$
We can write the exterior derivatives of these forms to be 
$$
de^1=0,\qquad de^2=\frac{\varphi'(r)}{\varphi(r)}e^1\wedge e^2,\qquad de^3=\frac{\varphi'(r)}{\varphi(r)}e^1\wedge e^3+\frac{\cot(\theta)}{\varphi(r)} e^2\wedge e^3.
$$
Given the caractherization of $\omega$, \eqref{carom},
one finds the explicit formulas
$$
\omega^1_{\; 2}=-\frac{\varphi'(r)}{\varphi(r)}e^2,\qquad \omega^1_{\; 3}=-\frac{\varphi'(r)}{\varphi(r)}e^3,\qquad
\omega^2_{\; 3}=-\frac{\cot(\theta)}{\varphi(r)}e^3.
$$
In terms of coordinates, this gives,
$$
\begin{array}{cc}
\omega_1^{\; ab} = 0  & \textrm{ for all }ab\\
\omega_2^{\; 12} = -\varphi'(r) & \\
\omega_2^{\; ab} = 0 & \textrm{ if } ab \neq 12 \textrm{ and }21\\
\omega_3^{\; 23} = - \cos \theta & \\
\omega_3^{\; 13} = - \varphi'(r) \sin \theta & \\
\omega_3^{\; ab}  = 0 & \textrm{ if } a=b \textrm{ or } ab = 12 \textrm{ or }21. \end{array}
$$

Therefore, after recalling \eqref{spinconn}, one can write 
\begin{eqnarray*}
\Gamma_1 & =& 0 \\  
\Gamma_2 &=& -\frac12 \varphi'(r) \gamma_1\gamma_2  \\
\Gamma_3  &=& -\frac12 \Big( \cos \theta \gamma_2 \gamma_3 + \varphi'(r) \sin\theta \gamma_1 \gamma_3\Big) 
\end{eqnarray*}
which in turns implies
\begin{eqnarray*}
D_1 &=& \partial_r \\
D_2 &=& \partial_\theta - \frac12 \varphi'(r) \gamma_1\gamma_2 \\
D_3 & = &\partial_\phi -\frac12 \Big( \cos \theta \gamma_2 \gamma_3 + \varphi'(r) \sin\theta \gamma_1 \gamma_3\Big).
\end{eqnarray*}

We have
$$
 e^j_{\; a} \gamma^a D_j  = \gamma^1 D_1 + \frac1{\varphi(r)} \gamma^2 D_2 + \frac1{\varphi(r)\sin\theta} \gamma^3 D_3.
$$
Using anti-commutation rules between the $\gamma^a$s, we get that $e^{j}_{\; a } \gamma^a D_j$ is equal to
$$
 \gamma^1 \partial_r + \frac1{\varphi(r)} \Big( \gamma^2 \partial_\theta + \gamma^1 \frac{\varphi'(r)}{2} \Big) + \frac1{\varphi(r) \sin\theta} \Big(\gamma^3 \partial_\phi + \gamma^2 \frac{\cos \theta}{2} + \gamma^1 \frac{\sin \theta \varphi'(r)}{2} \Big).
 $$
Rearranging the sum, we can rewrite equation \eqref{gendir} as
\begin{equation*}
\left[
\gamma^0\partial_t+\gamma^1\left(\partial_r+\frac{\varphi'(r)}{\varphi(r)}\right)+\frac1{\varphi(r)}
\left(\left(\partial_\theta+\frac{\cot\theta}2\right)\gamma^2+\frac1{\sin\theta}\partial_\phi\gamma^3\right)
-m\right]\psi=0.
\end{equation*}
It is helpful to rewrite the equation above in the Hamiltonian form: multiplying it times $i\gamma^0$ yields
$$
i\partial_t\psi=(\mathcal{D}+\gamma^0m)\psi
$$
where the Dirac operator is now written as
$$
\mathcal{D}=-i\gamma^0\gamma^1\left(\partial_r+\frac{\varphi'(r)}{\varphi(r)}\right)+\frac1{\varphi(r)}\left[\left(-i\partial_\theta-\frac{i\cot(\theta)}2\right)\gamma^0\gamma^2-\frac{i}{\sin\theta}\partial_\phi\gamma^0\gamma^3\right].
$$
To simplify the notations, we denote with
$$
\alpha^j:=\gamma^0\gamma^j,\qquad j=1,2,3;
$$
notice that these matrices satisfy now the anticommutation relation
$$
\{\alpha^j,\alpha^k\}=2\delta^{jk},\qquad \forall\: j,k=1,2,3.
$$
We also write
$$
\mathcal{D}_{\mathbb{S}^2}=\alpha^2\left(-i\partial_\theta-\frac{i\cot(\theta)}2\right)-\alpha^3\frac{i}{\sin\theta}\partial_\phi .
$$
 Putting things together, we have finally reached the following representation for the Dirac equation on a spherically symmetric manifold $(M,g)$ with a metric of the form \eqref{sphermet}
 \begin{equation}\label{dir2}
 i\partial_t\psi=(\mathcal{D}+\gamma^0m)\psi,\qquad \mathcal{D}=-i\alpha^1\left(\partial_r+\frac{\varphi'(r)}{\varphi(r)}\right)+\frac1{\varphi(r)}{\mathcal{D}_{\mathbb{S}^2}}.
 \end{equation}

\subsection{Diagonalization}\label{Diag}

We want to diagonalize this operator and put it in a more convenient form. For this, we use the existence in the physics literature, see \cite{abri}, of a diagonilazation of a "cousin" to $\mathcal D_{\mathbb{S}^2}$ that is, the Dirac operator on the sphere $\mathbb S^2$ : 
$$
i\hat \nabla = -i\sigma_1 \Big( \partial_\theta + \frac{\cot \theta}{2} \Big) - i \frac{\sigma_2}{\sin \theta} \partial_\phi .
$$
Writing $H_\varphi = m\beta - i\alpha^1 \Big( \partial_r + \frac{\varphi'}{\varphi} \Big) + \frac1{\varphi} \mathcal D_{\mathbb{S}^2}$, we get that $H_\varphi$ can be written in blocks as
$$
\begin{pmatrix} m & A \\A & -m \end{pmatrix}
$$
with 
$$
A = -i \sigma_1 \Big( \partial_r + \frac{\varphi'}{\varphi} \Big) + \frac1{\varphi} \Big( -i \sigma_2  \Big( \partial_\theta + \frac{\cot \theta }{2} \Big) -\frac{\sigma_3}{\sin \theta} \partial_\phi \Big) .
$$
The following permutation
\begin{eqnarray*}
\sigma_2 & \leftarrow & \sigma_1 \\
\sigma_3 & \leftarrow & \sigma_2 \\
\sigma_1 & \leftarrow & \sigma_3
\end{eqnarray*}
is equivalent in the gamma matrices framework to the permutation 
\begin{eqnarray*}
\alpha_2 & \leftarrow & \alpha_1 \\
\alpha_3 & \leftarrow & \alpha_2 \\
\alpha_1 & \leftarrow & \alpha_3 \\
\beta & \leftarrow & \beta ,
\end{eqnarray*}
which in turns corresponds to an orthogonal change of basis in $\C^4$. In other words, up to a rotation in $\C^4$, $\psi$ satisfies $i\partial_t \psi  = H_\varphi \psi$ with 
$$
H_\varphi =\begin{pmatrix} m & -i \sigma_3 \Big( \partial_r + \frac{\varphi'}{\varphi} \Big) + \frac1{\varphi}(-i\hat \nabla)\\
-i \sigma_3 \Big( \partial_r + \frac{\varphi'}{\varphi} \Big) + \frac1{\varphi}(-i\hat \nabla)& -m \end{pmatrix}.
$$

The operator $-i\hat \nabla$ diagonalises into (we refer to Section 2.3 in \cite{abri})
$$
- i\hat \nabla \Gamma_{j,m}^\pm = \pm \lambda_j \Gamma_{j,m}^\pm
$$
where $j \in \frac12 + \N$ and $m_j \in \frac12 + \Z$ with the constraint $-j \leq m_j \leq j$ and finally $\lambda_j = \frac12 + j$. 
Note that the parametrization in $j,m_j$ corresponds to the parametrization of the diagonalization of $\mathcal D_{\R^3}$ one can find in Thaller's book, as $j$ corresponds to the primary quantum number of total angular momentum and $m_j$ to the secondary one. 
One can choose such $\Gamma_{j,m_j}^\pm$ such that
$$
\Gamma_{j,m_j}^\pm = \pm i \sigma_3 \Gamma_{j,m_j}^\mp \Leftrightarrow -i\sigma_3 \Gamma_{j,m_j}^\pm = \pm \Gamma_{j,m_j}^\mp.
$$
And of course, they are chosen such that they form an orthogonal basis of $L^2$, that is
$$
\an{\Gamma_{j_1,m_{j_1}}^{\varepsilon_1}, \Gamma_{j_2,m_{j_2}}^{\varepsilon_2}} = \delta^{\varepsilon_1,\varepsilon_2} \delta_{j_1,j_2} \delta_{m_{j_1},m_{j_2}}.
$$
Set $\tilde H_\varphi = -i\sigma_3 \Big(\partial_r + \frac{\varphi'}{\varphi} \Big) + \frac1{\varphi} (-i\hat \nabla)$.  In other words, we have 
$$
H_\varphi = \begin{pmatrix} m & \tilde H_\varphi \\ \tilde H_\varphi & -m
            \end{pmatrix} .
$$
Let $E_{j,m_j}^\pm = \frac1{\sqrt 2} \Big( \Gamma_{j,m_j}^+ \pm \Gamma_{j,m_j}^-\Big)$. Now, we see how $\tilde H_\varphi$ acts on $ H^1(\varphi(r)^2dr) \otimes E_{j,m_j}^\pm $: we have for $f \in H^1(\varphi(r)^2dr) $ 
$$
\tilde H_\varphi fE_{j,m_j}^\pm = \Big( \mp \Big( \partial_r + \frac{\varphi'}{\varphi} \Big)f + \frac{\lambda_j}{\varphi} f \Big) E_{j,m_j}^\mp.
$$
The $E_{j,m_j}^\pm$ form an orthogonal basis of $L^2(\mathbb{S}^2, \C^2)$. We introduce
$$
F_{j,m_j}^- = \begin{pmatrix} E_{j,m_j}^- \\ 0 \end{pmatrix} , F_{j,m_j}^+ = \begin{pmatrix} 0 \\ E_{j,m_j}^+  \end{pmatrix} , G_{j,m_j}^+  = \begin{pmatrix} E_{j,m_j}^+ \\ 0 \end{pmatrix}, G_{j,m_j}^- = \begin{pmatrix} 0 \\ -E_{j,m_j}^-  \end{pmatrix}  .
$$
These form an orthogonal basis of $L^2(\mathbb S^2, \C^4)$. We see now how $H_\varphi$ acts on $H^1(\varphi(r)^2dr) \otimes F_{j,m_j}^\pm$. Let $f \in H^1(\varphi(r)^2dr) $, we have :
$$
H_\varphi f F_{j,m_j}^- = \begin{pmatrix} m fE_{j,m_j}^- \\\Big( \Big( \partial_r + \frac{\varphi'}{\varphi} \Big)f + \frac{\lambda_j}{\varphi} f \Big) E_{j,m_j}^+ \end{pmatrix} = mfF_{j,m_j}^- +\Big( \Big( \partial_r + \frac{\varphi'}{\varphi} \Big) f+ \frac{\lambda_j}{\varphi} f \Big) F_{j,m_j}^+.
$$
For the same reasons
$$
H_\varphi f F_{j,m_j}^+ = \Big(- \Big( \partial_r + \frac{\varphi'}{\varphi} \Big)f + \frac{\lambda_j}{\varphi}f\Big) F_{j,m_j}^- - m fF_{j,m_j}^+.
$$
Therefore, on the subspace $H^1(\varphi(r)^2dr) \otimes Vect ( F_{j,m}^-, F_{j,m}^+)$, $H_\varphi$ acts like 
\begin{equation}\label{raddir1}
h_{j,m,\lambda_j} = \begin{pmatrix} m & -\Big(\partial_t + \frac{\varphi'}{\varphi} \Big) + \frac{\lambda_j}{\varphi} \\ \Big(\partial_t + \frac{\varphi'}{\varphi} \Big) + \frac{\lambda_j}{\varphi} &  -m \end{pmatrix}.
\end{equation}
We call $\tilde{\mathcal H}_{j,m_j,\lambda_j}$ the subspace of $L^2(\mathbb{S}^2, \C^4)$ generated by $F_{j,m_j}^-,F_{j,m_j}^+$.

We see now how $H_\varphi$ acts on $H^1(\varphi(r)^2dr) \otimes G_{j,m_j}^\pm$. Let $f \in H^1(\varphi(r)^2dr) $, we have :
$$
H_\varphi f G_{j,m_j}^+ = \begin{pmatrix} m fE_{j,m_j}^+ \\\Big(- \Big( \partial_r + \frac{\varphi'}{\varphi} \Big)f + \frac{\lambda_j}{\varphi} f \Big) E_{j,m_j}^- \end{pmatrix} = mfG_{j,m_j}^+ +\Big( \Big( \partial_r + \frac{\varphi'}{\varphi} \Big) f- \frac{\lambda_j}{\varphi} f \Big) G_{j,m_j}^-.
$$
For the same reasons
$$
H_\varphi f G_{j,m_j}^- = \Big(- \Big( \partial_r + \frac{\varphi'}{\varphi} \Big)f - \frac{\lambda_j}{\varphi}f\Big) G_{j,m_j}^+ - m fG_{j,m_j}^-.
$$
Therefore, on the subspace $H^1(\varphi(r)^2dr) \otimes Vect ( G_{j,m_j}^+, G_{j,m_j}^-)$, $H_\varphi$ acts like 
\begin{equation}\label{raddir2}
h_{j,m_j,-\lambda_j} = \begin{pmatrix} m & -\Big(\partial_t + \frac{\varphi'}{\varphi} \Big) - \frac{\lambda_j}{\varphi} \\ \Big(\partial_t + \frac{\varphi'}{\varphi} \Big) -\frac{\lambda_j}{\varphi} &  -m \end{pmatrix}.
\end{equation}
We call $\tilde{\mathcal H}_{j,m_j,-\lambda_j}$ the subspace of $L^2(\mathbb{S}^2, \C^4)$ generated by $G_{j,m_j}^+,G_{j,m_j}^-$.

We now refer to \cite{abri} eq (62) p12 to get that up to a local rotation 
$$
R_1 = e^{i\sigma_2 \frac{\theta}{2}} e^{i\sigma_3 \frac{\phi}{2}},
$$
the maps $\Gamma_{j,m_j}^\pm$ belong to a combination of spherical harmonics of degree $j-\frac12, j+\frac12$. More precisely, we have
$$
R_1^* \Gamma_{j,m_j}^{\pm} = \frac1{\sqrt 2} \begin{pmatrix} \sqrt{\frac{j+m_j}{2j}}Y_{j^-,m_j^-} \pm \sqrt{\frac{j-m_j+1}{2j+2}} Y_{j^+,m_j ^-}\\ \sqrt{\frac{j-m_j}{2j}}Y_{j^-,m_j^+}\mp \sqrt{\frac{j+m_j+1}{2j+1}} Y_{j^+,m_j^+}\end{pmatrix}
$$
where $Y_{j,m_j}$ are the standard spherical harmonics, $j^\pm$ and $m_j^\pm$ stands respectively for $j\pm \frac12$ and $m_j\pm \frac12$. In terms of $E$, that means
$$
R_1^*E_{j,m_j}^+ = \begin{pmatrix} \sqrt{\frac{j+m_j}{2j}} Y_{j-,m_j^-} \\ \sqrt{\frac{j-m_j}{2j}}Y_{j^-,m_j^+} \end{pmatrix} \textrm{ and } R_1^* E_{j,m_j}^- = \begin{pmatrix} \sqrt{\frac{j-m_j + 1}{2j+2}}Y_{j^+,m_j^-} \\ -\sqrt{\frac{j+m_j+1}{2j+2}}Y_{j^+,m_j^+} \end{pmatrix}.
$$
Hence, we get that 
$$
\begin{pmatrix} Y_{j^+,m_j ^-}\\ 0 \end{pmatrix} \textrm{ and } \begin{pmatrix} 0\\ Y_{j^+,m_j^+}\end{pmatrix}
$$
are linear combinations of $R_1^* E_{j,m_j}^-$ and $R_1^* E_{j+1,m_j}^+$. Writing $\mathcal H_{j,m_j, k_j} = (R_1^* \oplus R_1^*) \tilde{\mathcal H}_{j,m_j,k_j}$ and keeping in mind the notations of the introduction \eqref{Hj}, \eqref{Pn}, we get as desired
$$
\mathcal S_n \subseteq \mathcal P_n.
$$
We note also that $\mathcal H_{j,m_j,k_j}$ corresponds de facto to the partial wave subspaces in Thaller's book. Note that $k_j$ corresponds to either plus or minus $\lambda_j = j + \frac12$. 

We write 
$$
\tilde \Pi_{j,m_j,k_j} : L^2(M, \C^4) \rightarrow L^2(\varphi^2(r)dr, \C^2)
$$
with value
$$
\tilde \Pi_{j,m_j,k_j} u = \left\lbrace{\begin{array}{cc}
                                  \begin{pmatrix} \an{F_{j,m_j,k_j}^-, u}_{L^2(M,\C^4)} \\ \an{F_{j,m_j,k_j}^+,u}_{L^2(M,\C^4)}\end{pmatrix} & \textrm{ if } k_j>0 \\
                                  \begin{pmatrix} \an{G_{j,m_j,k_j}^+, u}_{L^2(M,\C^4)} \\ \an{G_{j,m_j,k_j}^-,u}_{L^2(M,\C^4)}\end{pmatrix} & \textrm{ if } k_j<0   
                                 \end{array}} \right.
$$
and $\tilde \Pi_{j,m_j,k_j}^*$ its adjoint. Note that its image is $\tilde{\mathcal H}_{j,m_j,k_j}$. We write $\Pi_{j,m_j,k_j} = R_1 \oplus R_1 \tilde \Pi_{j,m_j,k_j}$. Its adjoint $R_1^* \oplus R_1^* \tilde \Pi_{j,m_j,k_j}^*$ has image $\mathcal H_{j,m_j,k_j}$.

Let $\psi$ be a solution to
$$
i\partial_t \psi = H_\varphi \psi = \sum_{j,m_j,k_j} \tilde \Pi_{j,m_j,k_j}^*h_{j,m_j,k_j} \tilde \Pi_{j,m_j,k_j} \psi .
$$
We introduce $u = (R_1^*\oplus R_1^*) \psi$, we have that $u$ solves 
$$
i\partial_t u = \sum_{j,m_j,k_j}  \Pi_{j,m_j,k_j}^*h_{j,m_j,k_j}  \Pi_{j,m_j,k_j} \psi .
$$

We prove Strichartz estimates and indeed work for $u$ as it is equivalent (even in the nonlinear model) to working on $\psi$, given the form of $R_1$.
 
\begin{remark} The fact that we have to apply this rotation $R_1$ is due to the fact that we have chosen to work with a different represention from the cartesian coordinates one can find in Thaller's book. The solution $u$ is actually the spinor we would work on if we had chosen the cartesian representation instead of the spherical one from the beginning. We chose the spherical one to avoid very heavy computations.\end{remark}

 \subsection{The Dirac equation on weighted spinors}
 The idea now is to reduce the study of the Dirac equation on a partial wave subspace to the one of a (radial) Dirac equation on $\mathbb{R}^3$ with a potential. 

We introduce the multiplication by $\sigma$ with
\begin{equation}\label{weightspin}
\sigma(r):=\frac{r}{\varphi(r)}.
\end{equation}
Straightforward calculations show that
$$
\sigma'(r)=\frac{\varphi(r)-r\varphi'(r)}{\varphi(r)^2}
$$
so that
\begin{equation}\label{sigma'sigma}
\frac{\sigma'}\sigma=\frac{1}r-\frac{\varphi'(r)}{\varphi(r)}.
\end{equation}
Therefore, the operator $\mathcal{D}$ in \eqref{dir2} acts on the weighted spinor (the angular part is left unchanged) as
\begin{equation}\label{newdir}
\sigma^{-1}\mathcal{D}\sigma =-i\alpha^1\left(\partial_r+\frac{1}{r}\right)+\frac1{\varphi(r)}{\mathcal{D}_{\mathbb{S}^2}}.
\end{equation}
The operators \eqref{raddir1}, \eqref{raddir2} are accordingly modified into the operator
\begin{equation}\label{raddirsigma}
h_{m_j,k_j}^\sigma=\left(\begin{array}{cc}m & -\frac{d}{dr}-\frac{1}r+\frac{k_j}{\varphi(r)} \\\frac{d}{dr}+\frac{1}{r}+\frac{k_j}{\varphi(r)} & -m\end{array}\right).
\end{equation}
In other words, the multiplication by $\sigma$ given by \eqref{weightspin} has turned the Dirac equation on $L^2(\varphi^2(r)dr) \otimes \mathcal H_{j,m_j,k_j}$ into the system on $L^2(r^2dr)^2$ : 
\begin{equation}\label{systg}
\begin{cases}
i\partial_t g_{k_j}^++mg_{k_j}^++ \left(-\frac{d}{dr}-\frac{1}r+\frac{k_j}{\varphi(r)}\right)g_{k_j}^-=0\\
i\partial_t g_{k_j}^--mg_{k_j}^-+ \left(\frac{d}{dr}+\frac{1}r+\frac{k_j}{\varphi(r)}\right)g_{k_j}^+=0.
\end{cases}
\end{equation}

Notice now that system above can be seen as the restriction of the Dirac equation on $\mathbb{R}^{1+3}$ to the $k_j$-th partial wave subspace perturbed with a (radial) potential (compare with (4.104) pag.125 in \cite{thaller} or \cite{lanlif} pag 108). This allows to rely on the well developed theory for potential perturbations of dispersive flows on $\mathbb{R}^n$ to obtain, quite straightforwardly, local in time Strichartz estimates. By summing and subtracting the angular Dirac operator ${\mathcal{D}_{\mathbb{S}^2}}$ with the weight $r^{-1}$, which is the one corresponding to the flat case, we can indeed rewrite \eqref{newdir} as
\begin{equation}\label{moddirac}
\sigma^{-1}\mathcal{D}\sigma =\mathcal{D}_{\mathbb{R}^3}+\left(\frac1{\varphi(r)}-\frac1{r}\right){\mathcal{D}_{\mathbb{S}^2}}.
\end{equation}
This suggests that we are dealing with a standard Dirac equation perturbed, on each partial wave subspace, with a radial  potential of the form $V(r)=k_j\left(\frac1{\varphi(r)}-\frac1{r}\right)\begin{pmatrix}
0 & 1 \\ 1 &0 
\end{pmatrix} $.

\begin{remark}
We should recall at this point that the action of the Dirac operator in the flat case perturbed by potentials of the form $V(x)=V(|x|)1_4+i\beta\alpha\cdot e_1 V_2(|x|)$, where $V_1$ and $V_2$ are two scalar functions, leaves invariant the partial wave subspaces defined above, and such action can be represented by the $2\times2$ matrix 
\begin{equation}\label{raddirV}
h^{pot}_{m_j,k_j}=\left(\begin{array}{cc}m+V_1(r) & -\frac{d}{dr}-\frac1r+\frac{k_j}{\varphi(r)}+V_2(r) \\\frac{d}{dr}+\frac1r+\frac{k_j}{\varphi(r)}+V_2(r) & -m+V_1(r)\end{array}\right).
\end{equation}
This is formula (4.129) in \cite{thaller}; we should stress that the additional term $\frac1r$ that we have above, and that is missing in \cite{thaller}, is due to the fact that here we have introduced a different weighted spinor, in order to deal with the metrics-type perturbative term. Therefore, if we take $V_1=0$ and $V_2=k_j\left(\frac1{\varphi(r)}-\frac1{r}\right)$ we have a structure as above.
\end{remark}

\begin{remark}
Relation \eqref{moddirac}, has a direct implication on the self-adjointness of the operator $\mathcal{D}$. Indeed, if we restrict it to any partial wave subspace, it defines a selfadjoint operator on $D(h_{m_j,k_j}^\sigma)=H^1(\mathbb{R},\mathbb{C}^4)$ due to Kato-Rellich Theorem as in fact the perturbative term $V(r)=k_j\left(\frac1{\varphi(r)}-\frac1{r}\right)$ is bounded in our assumptions {\bf (A1)}. Therefore $\sigma^{-1}\mathcal D \sigma$ is a self-adjoint operator in $\R^3$ with domain $H^1(\R^3)$. Given that for any test functions $u,v$,
$$
\an{\sigma^{-1}\mathcal D \sigma u, v}_{\R^3} = \an{\sigma \mathcal D \sigma u, v}_M
$$
we get that $\mathcal D$ is a self-adjoint operator in $M$ with domain $\sigma^{-1} H^1(\R^3) = H^1(M)$ as explained in the proof of Therorem \ref{teo1}. 
\end{remark}

\begin{remark}
Note that it would be possible to define a slightly more general form of weighted spinors \eqref{weightspin} by setting, for any $\kappa\in\N$,
\begin{equation}\label{weightspin2}
\sigma(r):=\frac{r^{\kappa+1}}{\varphi(r)}.
\end{equation}
This more general choice will not produce any significant advantage for the purpose of the present paper, and therefore we retrieve the specialized form with $\kappa=0$ of \eqref{weightspin}. Anyway, we mention the fact that it might be useful in view of proving global Strichartz estimates: indeed, by squaring the resulting system (that is, the analogue of \eqref{systg} for the new weighted spinors), one would obtain a system of decoupled Klein-Gordon equations on a higher space dimension $\mathbb{R}^N$ with $N>3$ perturbed with a radial potential that, in general, will have an inverse-square decay. This kind of dynamics have been widely investigated in literature (see e.g. \cite{burq}) and dispersive estimates are well known for them; there is thus the chance to adapt these results to the present setting, and this will be the object of future investigations. We mention that a similar point of view has been developed in \cite{danzha} in the different contest of the study of equivariant wave maps.
\end{remark}

 \section{Radial Strichartz estimates: proof of Theorem \ref{teo1}}

 This section is devoted to the proof of the weighted Strichartz estimates stated in Theorem \eqref{teo1}. The main idea will be to use the decomposition discussed in the previous section to reduce the problem to a Dirac equation on $\R^3$ perturbed with a potential, and then rely on Proposition \ref{boundpot}.

\begin{lemma}\label{lemma1} Let $V = \Big( \frac1{\varphi} - \frac1{r}) {\mathcal{D}_{\mathbb{S}^2}}$. Let $V_n$ be the restriction of $V$ to $\mathcal P'_n : =L^2(r^2dr) \otimes \mathcal P_n$. We have that $V_n$ is an endomorphism of $\mathcal P'_n$, it is continuous from $H^{s}(\R^3)$ to itself for any $s\in [0,1]$ and
$$
\|V_n\|_{H^{s} \rightarrow H^s} \lesssim \an n .
$$
\end{lemma}

\begin{remark}
All the norms appearing both in the statement above and in the next proof, when not differently specified, will be taken on $\R^3$.\end{remark}

\begin{proof}  We have that 
$$
V_n = \Big( \bigoplus_{j=n-1/2, m_j, k_j} V_{j,m_j,k_j} \Big) \oplus \Big( \bigoplus_{j=n+1/2, m_j, k_j} V_{j,m_j,k_j} \Big)
$$
where $V_{j,m_j,k_j}$ is the restriction to  $L^2(r^2dr) \otimes \mathcal H_{j,k_j,m_j}$ of $V_n$, that is the restriction to $L^2(r^2dr) \otimes \mathcal H_{j,k_j,m_j}$ of $V$.

The operator $V_{j,m_j,k_j}$ is represented by the matrix
$$
V_{j,m_j,k_j}= \Big( \frac1{\varphi(r)} - \frac1{r} \Big) k_j \begin{pmatrix} 0 & 1 \\1 & 0 \end{pmatrix},
$$
therefore
$$
\|V_{j,m_j,k_j}\|_{L^2 \rightarrow L^2} \leq \big\| \frac1{\varphi(r)} - \frac1{r} \big\|_{L^\infty} |k_j| 
$$
and
$$
\|V_{j,m_j,k_j}\|_{H^1 \rightarrow H^1} \leq \big\| \frac1{\varphi(r)} - \frac1{r} \big\|_{W^{1,\infty}} |k_j| .
$$
We have that 
$$
\frac1{\varphi(r)} - \frac1{r} = \frac{r-\varphi(r)}{r\varphi(r)}.
$$
Since as $r$ goes to $0$,
$$
\varphi(r) =  r+ o(r^2)
$$
we get that there exists $\eta> 0$ such that for all $r< \eta$
$$
|\varphi(r) - r| \leq r^2 \textrm{ and } \varphi(r) \geq \frac{r}{2}
$$
hence
$$
\Big| \frac1{\varphi(r)} - \frac1{r} \Big| \leq 2.
$$
For $r\geq \eta$, we use that $\varphi$ is bounded by below for $r\geq 1$  and that $\varphi$ is smooth and positive on $(0,\infty)$ to get
$$
\sup_{r\geq \eta }\Big| \frac1{\varphi(r)} - \frac1{r} \Big| <\infty.
$$
Hence 
$$
\big\| \frac1{\varphi(r)} - \frac1{r} \big\|_{L^\infty} < \infty.
$$
What is more 
$$
\partial_r \Big(\frac1{\varphi(r)} - \frac1{r}\Big)  = \frac1{r^2} - \frac{\varphi'(r)}{\varphi^2(r)}  = \frac{\varphi^2(r) - r^2\varphi'(r)}{r^2\varphi^2(r)}.
$$
Since when $r$ goes to $0$, $\varphi(r)^2 - r^2\varphi'(r) = O(r^4)$ and since $\varphi^2(r) = r^2 + o(r^3)$ we get that there exists $C,\eta$ such that for all $r<\eta$,
$$
| \varphi^2(r) - r^2\varphi'(r) | \leq C r^4 \textrm{ and } r^2 \varphi^2(r) \geq \frac{r^4}{2}
$$
from which we get
$$
\Big|\partial_r \Big(\frac1{\varphi(r)} - \frac1{r}\Big)\Big| \leq 2C.
$$
We use that $\frac1{\varphi}$ and $\varphi'/\varphi$ are bounded on $[\eta,\infty[$ to get
$$
\sup_{r\geq \eta}\Big|\partial_r \Big(\frac1{\varphi(r)} - \frac1{r}\Big)\Big| < \infty.
$$
Therefore,
$$
\big\| \frac1{\varphi(r)} - \frac1{r} \big\|_{W^{1,\infty}} < \infty.
$$
We get that there exists $C_\varphi$ such that
\begin{eqnarray*}
\|V_{j,m_j,k_j}\|_{L^2\rightarrow L^2} &\leq & C_\varphi |k_j| \\
\|V_{j,m_j,k_j}\|_{H^1\rightarrow H^1} &\leq & C_\varphi |k_j| 
\end{eqnarray*}
and by interpolation for all $s \in [0,1]$,
$$
\|V_{j,m_j,k_j}\|_{H^s\rightarrow H^s} \leq C_\varphi |k_j|
$$
and since $|k_j| = j+\frac12 = n$ or $n+1$ we get by summing up
$$
\|V_n\|_{H^s \rightarrow H^s} \leq C_\varphi \an n.
$$
\end{proof}

\begin{lemma}\label{lemma2} Let $s\in [0,1]$ and $\sigma$ be given by \eqref{weightspin}. We have that the multiplication by $\sigma $ is continuous from $H^s(\R^3)$ to $H^s(M)$ and that the multiplication by $\sigma^{-1}$ is continuous from $H^s(M)$ to $H^s(\R^3)$.
\end{lemma}

\begin{proof} For $s=0$, the multiplication by $\sigma$ is an isometry from $L^2(\R^3)$ to $L^2(M)$ which implies both continuities.

For $s=1$, we have for $u \in H^1(\R^3)$
$$
\|\sigma u\|_{\dot H^1(M)} = \|\grad (\sigma u) \|_{L^2(M)} \leq \|(\grad \sigma) u\|_{L^2(M)} + \|\sigma \grad u\|_{L^2(M)}.
$$
We have that
$$
\|\sigma \grad u\|_{L^2(M)} = \|\grad u\|_{L^2(\R^3)} = \|u\|_{\dot H^1(\R^3)}.
$$
What is more 
$$
\|(\grad \sigma) u\|_{L^2(M)} = \|\frac{\sigma'}{\sigma} u\|_{L^2(\R^3)}.
$$
We have 
$$
\frac{\sigma'}{\sigma } = \frac1{r} - \frac{\varphi'}{\varphi} = \frac{\varphi - r\varphi'}{r\varphi}.
$$
We use that as $r$ goes to $0$
$$
\frac{\varphi - r\varphi'}{r\varphi} = o(1)
$$
and that for any $\eta > 0$, for all $r\geq \eta$,
$$
\Big| \frac{\sigma'}{\sigma } \Big| \leq \frac1{\eta} + \|\varphi'/\varphi\|_{L^\infty([\eta,\infty[)}
$$
to get that
$$
\frac{\sigma'}{\sigma} \in L^\infty
$$
and obtain that the multiplication by $\sigma$ is continuous from $H^1(\R^3)$ to $H^1(M)$ and by interpolation from $H^s(\R^3)$ to $H^s(M)=$ for $s \in [0,1]$.

For $u \in H^1(M)$, we have
$$
\|\sigma^{-1}u\|_{\dot H^1(\R^3)} =  \|\grad ( \sigma^{-1} u)\|_{L^2(\R^3)} \leq \|\frac{\sigma'}{\sigma^2}u\|_{L^2(\R^3)} + \|\sigma^{-1}\grad u\|_{L^2(\R^3)}.
$$
We have 
$$
\|\sigma^{-1}\grad u\|_{L^2(\R^3)} = \|\grad u \|_{L^2(M)} = \|u\|_{\dot H^1(M)} \leq \|u\|_{H^1(M)}
$$
and
$$
\|\frac{\sigma'}{\sigma^2}u\|_{L^2(\R^3)} = \|\frac{\sigma'}{\sigma}u \|_{L^2(M)}
$$
and since $\frac{\sigma'}{\sigma}$ belongs to $L^\infty$, we get
$$
\|\sigma^{-1} u\|_{H^1(\R^3)} \leq C_\varphi \|u\|_{H^1(M)}
$$ 
hence the multiplication by $\sigma^{-1}$ is continuous from $H^1(M)$ to $H^1(\R^3)$ and from $L^2(M)$ to $L^2(\R^3)$ 
and by interpolation from $H^s(M)$ to $H^s(\R^3)$ for $s\in [0,1]$.
\end{proof}

\begin{lemma}\label{lemma3} Let $V = \Big( \frac1{\varphi} - \frac1{r}) {\mathcal{D}_{\mathbb{S}^2}}$. The, the flow $e^{it(\mathcal D_{\R^3} + m\beta + V)}$ is continuous from $H^s(\R^3)$ to $\mathcal C(\R,H^s(\R^3))$ for any $s\in [0,1]$.
\end{lemma}

\begin{proof} Indeed, we have that
$$
\sigma^{-1} \mathcal D \sigma = \mathcal D_{\R^3} + m\beta + V
$$
therefore
$$
S(t) : = e^{it(\mathcal D_{\R^3} + m\beta + V)}  = e^{it(\sigma^{-1}\mathcal D \sigma)} = \sigma^{-1}e^{it\mathcal D}\sigma.
$$
We deduce 
$$
\|S(t)\|_{H^s(\R^3) \rightarrow \mathcal C(\R,H^s(\R^3))} \leq \|\sigma^{-1}\|_{H^{s}(M)\rightarrow H^s(\R^3)} \|e^{it\mathcal D}\|_{H^s(M) \rightarrow \mathcal C(\R,H^s(M))} \|\sigma\|_{H^s(\R^3) \rightarrow H^s(M)}.
$$
The last lemma ensures that $\|\sigma\|_{H^s(\R^3) \rightarrow H^s(M)}$ and $\|\sigma^{-1}\|_{H^{s}(M)\rightarrow H^s(\R^3)}$ are bounded. 

We recall that the scalar curvature is bounded and that
$$
\mathcal D^2 = -D^jD_j + \frac14 \mathcal R_h + m^2.
$$
Therefore,
$$
\Lambda : = \mathcal D^2 + \frac12 \|\mathcal R_h\|_{L^\infty(M)} - m^2
$$
is a positive operator that commutes with $\mathcal D$ and such that 
$$
\|\Lambda u\|_{L^2} \sim \|u\|_{H^2}.
$$
From the conservation of the $L^2$ norm under the flow of the Dirac equation we get for all $t\in \R$,
$$
\|e^{it \mathcal D} \|_{H^2\rightarrow H^2} \sim \|\Lambda e^{it\mathcal D}\Lambda^{-1}\|_{L^2\rightarrow L^2} = \|e^{it \mathcal D}\|_{L^2\rightarrow L^2} = 1.
$$
By interpolation, we get that for any $s \in [0,2]$, $t\mapsto e^{it\mathcal D}$ belongs to $H^s \rightarrow \mathcal C(\R, H^s)$.

We use this continuity of the flow $e^{it\mathcal D}$ to get
$$
\|S(t)\|_{H^s(\R^3) \rightarrow \mathcal C(\R,H^s(\R^3))} \lesssim \|\sigma^{-1}\|_{H^{s}(M)\rightarrow H^s(\R^3)} \|\sigma\|_{H^s(\R^3) \rightarrow H^s(M)} <\infty.
$$
\end{proof}

\begin{proof}(of Theorem \ref{teo1}) We focus on $m=0$, the case $m\neq 0$ being analogous. We write $S_n$ the restriction to $\mathcal P'_n = L^2(r^2dr ) \otimes \mathcal P_n$ of $S : t \mapsto S(t) = e^{it(\mathcal D_{\R^3} + m\beta + V)}$ We have that
$$
S_n(t) = e^{it(\mathcal D_{\R^3} + m\beta + V_n)}.
$$
Let $p> 2$, $q\geq 2$ such that $\frac1{p} + \frac1{q} = \frac12$. Let $I$ be a bounded interval of $\R$.
We have that 
$$
\|V_n\|_{H^{2/p} \rightarrow H^{2/p}} \leq C_\varphi \an n
$$
and 
$$
\|S_n\|_{H^{2/p}(\R^3) \rightarrow \mathcal C(\R, H^{2/p}(\R^3))} \leq \|S\|_{H^{2/p}(\R^3) \rightarrow \mathcal C(\R, H^{2/p}(\R^3))} \leq C_\varphi .
$$
Hence by Lemma \ref{boundpot}, we get for all $v_0 \in \mathcal P'_n \cap H^{2/p}(\R^3)$,
$$
\|S(t) v_0\|_{L^p_t(I)L^q(\R^3)} \leq C_\varphi \an n \|v_0\|_{H^{2/p}(\R^3)}.
$$

Let $f \in L^q(\R^3)$, we have that $\sigma^{2/q}f$ belongs to $L^q(M)$ and 
$$
\|\sigma^{2/q} f\|_{L^q(M)}^q = \int \sigma^{2} |f(r\omega)|^q \varphi^2(r)drd\omega = \int \sigma^{2- 2} |f(r\omega)|^q r^2drd\omega = \|f\|_{L^q(\R^3)}.
$$

Let $u_0 \in H^{2/p}(M)\cap L^2(\varphi^2(r)dr) \otimes \mathcal P_n$, we use that $e^{it\mathcal D} = \sigma S(t) \sigma^{-1}$ to get
$$
\|\sigma^{2/q -1}e^{it\mathcal D}u_0\|_{L^p_t(I)L^q(M)} =  \|\sigma^{2/q} S(t) \sigma^{-1} u_0\|_{L^p_t(I)L^q(M)}
$$
We use the last remark to get
$$
\|\sigma^{2/q -1}e^{it\mathcal D}u_0\|_{L^p_t(I)L^q(M)} =  \| S(t) \sigma^{-1} u_0\|_{L^p_t(I)L^q(\R^3)}
$$
We have that $\sigma^{-1}u_0 \in \mathcal P'_n \cap H^{2/p} (\R^3)$ because $\sigma^{-1}$ is an isometry from $L^2(\varphi^2(r)dr) \otimes \mathcal P_n$ to $\mathcal P'_n$ and is continuous from $H^{2/p}(M)$ to $H^{2/p}(\R^3)$.

Hence 
$$
\|\sigma^{2/q -1}e^{it\mathcal D}u_0\|_{L^p_t(I)L^q(M)} \leq C_\varphi \an n  \|  \sigma^{-1} u_0\|_{H^{2/p}(\R^3)}
$$
from which we get
$$
\|\sigma^{2/q -1}e^{it\mathcal D}u_0\|_{L^p_t(I)L^q(M)} \leq C_\varphi \an n  \|  u_0\|_{H^{2/p}(M)}
$$
\end{proof}

\section{Strichartz estimates with angular regularity: proof of Corollary \ref{cor1} in }

In this section we show how to extend the Strichartz estimates on partial wave subspaces proved in Theorem \ref{teo1} to general initial data; the main tool is given by the Littlewood-Paley decomposition on the sphere. We will deal separately with the case $m=0$ and $m\neq0$.

\subsection{The case $m=0$}
First, by interpolating Theorem \ref{teo1} and the conservation of the $L^2$ norm we get the following lemma.

\begin{lemma}\label{lem-interp1} Let $p>2$, $q\geq 2$ such that $\frac1{p} + \frac1{q} = \frac12$. Let $\varepsilon>0$. Let $I$ be a bounded interval of $\R$. There exists $C$ such that for all $u_0 \in L^2(\varphi^2(r) dr) \otimes \mathcal P_n\cap H^{2/p}(M)$ the solution $u$ to the linear Dirac equation \eqref{dir2} with initial datum $u_0$ satisfies 
$$
\|u\Big( \frac\varphi{r} \Big)^{2/p}\|_{L^p_t(I)L^q(M)} \leq C \an{n}^{2/p + \varepsilon} \|u_0\|_{H^{2/p}(M)}.
$$
\end{lemma}

\begin{proof} If $\theta := \frac2{p} + \varepsilon \geq 1$ then this is a direct consequence of Theorem \ref{teo1}. Otherwise, let $p_1 = \theta p$ and $\frac1{q_1} = \frac12 - \frac1{p_1} = \frac1\theta \Big( \frac1{q} - \frac{1-\theta}{2}\Big)$.

We have $\frac1{p} = \frac{1-\theta}{\infty} +  \frac{\theta }{p_1}$ and $\frac1{q} = \frac{1-\theta}{2} + \frac{\theta}{q_1}$. We can interpolate the conservation of the $L^2$ norm
$$
\|u\|_{L^\infty_t(I)L^2(M)} \leq  \|u_0\|_{L^2(M)}
$$
and
$$
\|u\Big( \frac\varphi{r} \Big)^{2/p_1}\|_{L^{p_1}_t(I)L^{q_1}(M)} \leq C \an{n} \|u_0\|_{H^{2/p_1}(M)}
$$
to get
$$
\|u\Big( \frac\varphi{r} \Big)^{2/p}\|_{L^p_t(I)L^q(M)} \leq C \an{n}^\theta \|u_0\|_{H^{2/p}(M)}
$$
which concludes the proof. \end{proof}

By summing over the decomposition in $\mathcal P_n$ and using Littlewood-Paley decomposition, we thus get the following lemma.

\begin{lemma}\label{lem-l2} Let $p>2$, $q\geq 2$ such that $\frac1{p} + \frac1{q} = \frac12$. Let $a,b > 0$ such that $b > \frac4{p}$ and $\frac{2}{pa} + \frac2{pb} < 1$. Let $I$ be a bounded interval of $\R$. There exists $C$ such that for all $u_0 \in H^{a,b}$ the solution $u$ to the linear Dirac equation \eqref{dir2} with initial datum $u_0$ satisfies 
$$
\|u\Big( \frac\varphi{r} \Big)^{2/p}\|_{L^p_t(I)L^q(M)} \leq C \|u_0\|_{H^{a,b}}.
$$
\end{lemma}

\begin{proof} We use Littlewood-Paley theory on the sphere. We refer essentially to Chapter 3.4 in \cite{daixu} and in particular Theorem 3.4.2. We recall that $\mathcal S_n$ are the spherical harmonics of degree $n$ and that $\mathcal S_n$ is included in $\mathcal P_n$. For $j\in \N$, we introduce $\pi_j$ be the orthogonal projection over 
$$
\tilde{\mathcal S}_j =\bigoplus_{n\in [2^j,2^{j+1}[}\mathcal S_n
$$ 
and $p_j$ be the orthogonal projection over 
$$
\tilde{\mathcal P}_j = \bigoplus_{n\in [2^j,2^{j+1}[}\mathcal P_n.
$$
By convention, $\tilde{\mathcal S}_{-1} = \mathcal S_0$ and $\pi_{-1}$ is the projection on $\tilde{\mathcal S}_{-1}$ and we define similarly $\tilde{\mathcal P}_{-1} = \mathcal P_0$, and $p_{-1}$ to be the orthogonal projection on $\tilde{\mathcal P}_{-1}$.

We have for $u_0 \in H^{a,b}$,
$$
\|u\|_{L^p_t(I)L^q(M)}= \big \| \; \|u(t,r)\|_{L^q(\mathbb{S}^2)}\big\|_{L^p_t(I)L^q(\varphi^2dr)}.
$$
Since $q\geq 2$, by Littlewood-Paley theory, we get
$$
\|u\|_{L^p_t(I)L^q(M)} \lesssim \big \| \; \Big( \sum_{j \in \{-1\}\cup \N} \|\pi_j u(t,r)\|_{L^q(\mathbb{S}^2)}^2 \Big)^{1/2}\big\|_{L^p_t(I)L^q(\varphi^2dr)}.
$$
We use that $\mathcal S_n$ is included in $\mathcal P_n$ and thus $\tilde{\mathcal S}_j \subseteq \tilde{\mathcal P}_j$ to get
$$
\pi_j u(t,r) = \pi_j p_j u(t,r).
$$
We also recall thanks to Littlewood-Paley theory that
$$
\|\pi_j p_j u(t,x)\|_{L^q(\mathbb{S}^2)} \leq \big \| \Big( \sum_k |\pi_k p_j u(t,x)|^2\Big)^{1/2}\big\|_{L^q(\mathbb{S}^2)} \lesssim \|p_j u(t,r)\|_{L^q}.
$$
Since $p$ and $q$ are bigger than $2$, by Minkowski inequality, we have 
$$
\|u\|_{L^p_t(I)L^q(M)} \lesssim   \Big( \sum_{n \in \N} \|p_j u(t,r)\|_{L^p_t(I)L^q(M)}^2 \Big)^{1/2}.
$$
We use the last lemma to get, for any $\varepsilon >0$, 
$$
\|u\|_{L^p_t(I)L^q(M)} \lesssim   \Big( \sum_{n \in \N} \an{2^j}^{4/p + \varepsilon}\|p_n u_0(r)\|_{H^{2/p}(M)}^2 \Big)^{1/2}.
$$
Let $\Lambda_r = \sqrt{1 - \frac1{\varphi^2(r)} \partial_r \varphi^2(r) \partial_r}$, we have 
$$
\|p_j u_0(r)\|_{H^{2/p}}^2  \leq \an{2^j}^{4/p} \|u_0\|_{L^2(M)}^2 + \|\Lambda_r^{2/p}u_0\|_{L^2(M)}^2.
$$
Since 
$$
x^{4/p + \varepsilon } y^{4/p} \leq x^{2b}+y^{2a}
$$
as soon as $\frac1{b}\Big( 2/p + \varepsilon/2) + \frac{1}{a}2/p \leq 1$, we get that for all $b > 4/p$ and 
$
\frac2{p}\Big( \frac1{a} + \frac1{b}\Big) < 1
$
we have
$$
\|u\|_{L^p_t(I)L^q(M)}\lesssim   \Big( \sum_{j \in \{-1\}\cup \N} \|p_j u_0\|_{H^{a,b}}^2 \Big)^{1/2},
$$
and since $\mathcal P_n = \mathcal H_{n-1/2} + \mathcal H_{n+1/2}$ and $L^2(\mathbb{S}^2) = \bigoplus_j \mathcal H_j$, we get
$$
\|u\|_{L^p_t(I)L^q(M)} \lesssim \|u_0\|_{H^{a,b}}.
$$

\end{proof}

\subsection{The case $m\neq 0$}

One advantage with respect to the case $m=0$ is that we have the end point : for all $u_0 \in L^2(\varphi(r)^2dr) \otimes \mathcal P_n \cap \mathcal H^{1/2}(M)$, the solution $u$ of the linear Dirac equation satisfies 
$$
\|\Big(\frac{\varphi(r)}{r}\Big)^{2/3} u\|_{L^2_t(\R)L^6(M)} \lesssim \an n \|u_0\|_{H^{1/2}}.
$$
Hence, by interpolation Theorem \ref{teo1} and the conservation of the $L^2$ norm we get the following lemma.

\begin{lemma}\label{lem-interp2} Let $p\geq 2$, $q\geq 2$ such that $\frac2{p} + \frac3{q} = \frac32$. Let $I$ be a bounded interval of $\R$. There exists $C$ such that for all $u_0 \in  L^2(\varphi^2(r) dr) \otimes \mathcal P_n\cap H^{1/p}(M)$ the solution $u$ to the linear Dirac equation \eqref{dir2} with initial datum $u_0$ satisfies 
$$
\|u\Big( \frac\varphi{r} \Big)^{1- \frac2{q}}\|_{L^p_t(I)L^q(M)} \leq C \an{n}^{2/p } \|u_0\|_{H^{1/p}}.
$$
\end{lemma}

By summing over the decomposition in $\mathcal P_n$ and using Littlewood-Paley decomposition, we thus get the following lemma.

\begin{lemma}\label{lem-l2mnot0} Let $p\geq 2$, $q\geq 2$ such that $\frac2{p} + \frac3{q} = \frac32$. Let $a,b > 0$ such that $b \geq \frac3{p}$ and $\frac{1}{pa} + \frac2{pb} \leq 1$. Let $I$ be a bounded interval of $\R$. There exists $C$ such that for all $u_0 \in H^{a,b}$ the solution $u$ to the linear Dirac equation \eqref{dir2} with initial datum $u_0$ satisfies 
$$
\|u\Big( \frac\varphi{r} \Big)^{1-2/q}\|_{L^p_t(I)L^q(M)} \leq C \|u_0\|_{H^{a,b}}.
$$
\end{lemma}

\begin{proof} We use the same notations as in the proof of Lemma \ref{lem-l2}. We have by definition:
$$
\|u\|_{L^p_t(I)L^q(M)} = \big \| \; \|u(t,r)\|_{L^q(\mathbb{S}^2)}\big\|_{L^p(I)L^q(\varphi^2dr)}.
$$
Since $q\geq 2$, we get by Littlewood-Paley theory
$$
\|u\|_{L^p_t(I)L^q(M)}\lesssim \big \| \; \Big( \sum_{j \in \{-1\}\cup\N} \|\pi_j u(t,r)\|_{L^q(\mathbb{S}^2)}^2 \Big)^{1/2}\big\|_{L^p(I)L^q(\varphi^2dr)}.
$$
We use that $\mathcal S_n$ is included in $\mathcal P_n$ to get
$$
\pi_j u(t,r) = \pi_j p_j u(t,r).
$$
We also recall thanks to Littlewood Paley theory that
$$
\|\pi_j p_j u(t,x)\|_{L^q(\mathbb{S}^2)}  \lesssim \|p_j u(t,r)\|_{L^p}.
$$
Since $p$ and $q$ are bigger than $2$, by Minkowski inequality, we have
$$
\|u\|_{L^p_t(I)L^q(M)} \lesssim   \Big( \sum_{j \in \{-1\}\cup\N} \|p_j u(t,r)\|_{L^p_t(I)L^q(M)}^2 \Big)^{1/2}.
$$
We use Lemma \ref{lem-interp2} to get
$$
\|u\|_{L^p_t(I)L^q(M)}\lesssim   \Big( \sum_{j \in \{-1\}\cup\N} \an{2^j}^{4/p }\|p_n u_0(r)\|_{H^{1/p}}^2 \Big)^{1/2}.
$$
We have 
$$
\|p_n u_0(r)\|_{H^{1/p}}^2  \leq \an{2^j}^{2/p} \|u_0\|_{L^2(M)}^2 + \|\Lambda_r^{1/p}u_0\|_{L^2(M)}^2.
$$
Since 
$$
x^{4/p } y^{2/p} \leq x^{2b}+y^{2a}
$$
as soon as $\frac2{pb} + \frac{1}{ap} \leq 1$, we get that for all $b \geq  3/p$ and 
$$
\frac1{p}\Big( \frac1{a} + \frac2{b}\Big) \leq 1
$$
we have
$$
\|u\|_{L^p_t(I)L^q(M)} \lesssim   \Big( \sum_{j \in \{-1\}\cup\N} \|p_j u_0\|_{H^{a,b}}^2 \Big)^{1/2}.
$$
and since $\mathcal P_n = \mathcal H_{n-1/2} + \mathcal H_{n+1/2}$ and $L^2(\mathbb{S}^2) = \bigoplus_j \mathcal H_j$, we get
$$
\|u\|_{L^p_t(I)L^q(M)} \lesssim \|u_0\|_{H^{a,b}}.
$$

\end{proof}

\section{A nonlinear application}

This section will be devoted to the proofs of Theorem \ref{teo3} and Corollary \ref{corsolmod}, which will be based on a standard contraction argument relying on Strichartz estimates. In this whole section, we assume $\inf \frac{\varphi(r)}{r} > 0$.

\subsection{Proof of Theorem \ref{teo3}}

Let us begin with the following Lemma.
\begin{lemma}\label{lem-striwithderm0} Let $p>2$, $q\geq 2$ such that $\frac1{p}  +\frac1{q} = \frac12$. Take $\kappa \in [0,\frac4{q})$ and $a,b>\kappa$ such that $\frac2{p(a-\kappa)} + \frac2{p(b-\kappa)} < 1$, $b-\kappa>\frac4{p}$. Let $I$ be a bounded interval of $\R$. There exists $C$ such that for all $u_0 \in H^{a,b}$ the solution $u$ to the Dirac equation with initial datum $u_0$ satisfies :
$$
\big\|u\|_{L^p_t(I)W^{\kappa,q}(M)} \leq C \|u_0\|_{H^{a,b}}
$$
\end{lemma}

\begin{proof} The case $\kappa = 0$ is a consequene of Corollary \ref{cor1} and of the existence of $\alpha$ such that $\varphi(r) \geq \alpha r$. Assume $\kappa > 0$. Take $\theta = \frac{\kappa}{2} \in (0,1)$. Define $p_1$, $q_1$, $a_1$ and $b_1$ such that
\begin{eqnarray*}
p_1 &=& (1-\theta) p  \\
\frac{1}{q_1} + \frac1{p_1} &=& \frac12 \\
a_1 &=& (1-\theta)^{-1}(a-\kappa) \\
b_1 &=& (1-\theta)^{-1} (b-\kappa).
\end{eqnarray*}
We have $p_1 > \Big( 1-\frac2{q}\Big)p = 2$ and 
$$
\frac1{a_1} + \frac1{b_1} = (1-\theta) \Big( \frac{1}{a-\kappa} + \frac1{b-\kappa}\Big) < (1-\theta) \frac{p}{2} = \frac{p_1}{2}
$$
and finally
$$
b_1  = (1-\theta)^{-1}(b-\kappa) > (1-\theta)^{-1}\frac{4}{p} = \frac4{p_1}.
$$
Thus we have 
\begin{equation}\label{intermstri}
\|e^{it\mathcal D}u_0\|_{L^{p_1}_t(I)L^{q_1}(M)} \lesssim \|u_0\|_{H^{a_1,b_1}}.
\end{equation}

What is more, we have 
\begin{eqnarray*}
\frac1{p} &=& \frac{1-\theta}{p_1} + \frac{\theta}{\infty} \\
\frac1{q} &=& \frac{1-\theta}{q_1} + \frac\theta{2} \\
\sigma  &=& \theta 2 \\
a &=& (1-\theta)a_1 + \theta 2 \\
b&=& (1-\theta)b_1 + \theta 2
\end{eqnarray*}
therefore, we can interpolate \eqref{intermstri} and 
$$
\|e^{it\mathcal D}u_0\|_{L^\infty_t(I)H^2(M)} \lesssim \|u_0\|_{H^{2,2}(M)}
$$
to get
$$
\|e^{it\mathcal D}u_0\|_{L^p_t(I)W^{\sigma,q}(M)} \lesssim \|u_0\|_{H^{a,b}}.
$$
\end{proof}

\begin{remark} When $a < 2$ interpolating the continuity of $e^{it\mathcal D}$ from $H^2$ to $\mathcal C(\R,H^2)$ and from $H^{0,b_1}$ to $\mathcal C(\R, H^{0,b_1}) $ ($\mathcal D$ and ${\mathcal{D}_{\mathbb{S}^2}}$ commute), we get the continuity of $e^{it\mathcal D}$ from $H^{a,b}$ to $\mathcal C(\R, H^{a,b})$.
\end{remark}

\begin{proof}[Proof of Theorem \ref{teo3}. ] Let $s_1 = \frac32 - \frac3{r'}$ and $a,b > s_1$ satisfying :
$$
\frac2{r'} \Big(\frac1{a-s_1} + \frac1{b-s_1}\Big) < 1 \; ,\; b> \frac1{r'} + \frac32.
$$
Let $f : [r',P) \rightarrow \R_+$ with $P = \frac6{3-2a}$ if $a< \frac32$, $P=\infty$ otherwise, defined as
$$
f(p) = \frac2{p} \Big(\frac1{a-s_1(p)} + \frac1{b-s_1(p)}\Big)
$$
such that $s_1(p ) = \frac32 - \frac3{p}$. Since $a$ and $b$ are strictly bigger than $s_1$, the map $f$ is continuous at $r'$ and besides $p \mapsto s_1(p)$ is continuous on $[r',\infty)$. Thus there exists $p > r'$ such that
$$
\frac2{p} \Big(\frac1{a-s_1(p)} + \frac1{b-s_1(p)}\Big) < 1 
$$
and $b> \frac1{p} + \frac32$.

This implies that $b-s_1(p) = b-\frac32 + \frac3{p} > \frac4{p}$. By continuity of 
$$
\kappa \mapsto \Big( \frac1{a-\kappa} + \frac1{b-\kappa}, b-\kappa\Big)
$$
at $s_1(p)$, we get that there exists $\kappa \in (s_1(p), \frac43 s_1(p))$ such that
$$
\frac2{p} \Big(\frac1{a-\kappa} + \frac1{b-\kappa}\Big) < 1 \textrm{ and } b-\kappa > \frac4{p}.
$$
Let $q$ be given by $ \frac1{q} + \frac1{p} = \frac1{2}$ (which is possible since $p> r'\geq 2$), we have $q\geq 2$ and $s_1(p) = \frac3{q}$, that is $\kappa \in (\frac3{q}, \frac4{q})$.

Denoting $S(t)$ the flow of the linear Dirac equation we get the Strichartz estimate:
$$
\big\|S(t) u_0\|_{L^p(I,W^{\kappa,q}(M))} \lesssim \|u_0\|_{H^{a,b}}.
$$
For $T \in (0,1]$ write
$$
X (T) = \mathcal C([-T,T], H^{a,b}) \cap L^p([-T,T], W^{\kappa,q}(M))
$$
By Strichartz inequality, since $[-T,T] \subseteq [-1,1]$, we have that there exists a constant $C$ independent on $T$ such that for all $u_0 \in H^{a,b}$
$$
\|S(t)u_0\|_{X(T)} \leq C \|u_0\|_{H^{a,b}}.
$$

Let $R = \|u_0\|_{H^{a,b}}$, we prove that $B = B_{X(T)}(0,2CR)$ is stable under 
$$
A_{u_0}  : u \mapsto \Big( t \mapsto S(t) u_0  - i\int_{0}^t S(t-\tau) |u(\tau)|^r u(\tau) d\tau \Big)
$$
where $|u|^2$ is either the density of mass $|u|^2 = \an{u,u}_{\C^4}$ or the density of charge $|u|^2 = \an{u,\beta u}_{\C^4}$ if $T$ is small enough.

What is more, we prove that $A_{u_0}$ is contracting on $B$ if $T$ is small enough.

Let $u \in B$, we have 
$$
\|A_{u_0}(u)\|_{X(T)} \leq C \|u_0\|_{H^{a,b}} + C \||u|^r u\|_{L^1_T H^{a,b}}.
$$
Notice that in the course of this proof we will use the more convenient notation $L^p_T=L^p_t([-T,T])$.
We use the Leibniz rule to get
$$
\||u|^r u\|_{ H^{a,b}} \lesssim \|u\|_{L^\infty}^r \|u\|_{H^{a,b}}.
$$
We then use the fact that $\kappa > \frac3{q}$ and Sobolev's estimates to get
$$
\|u\|_{L^\infty} \lesssim \|u\|_{W^{\kappa,q}(M)}.
$$
To sum up, we get that
$$
\|A_{u_0}(u)\|_{X(T)} \leq C \|u_0\|_{H^{a,b}} + C_1 \|\; \|u\|_{W^{\kappa,q}(M)}^r \|u\|_{H^{a,b}}\|_{L^1_T}.
$$
Finally, we use that $p > r' \geq r$ to use H\"older inequality :
$$
\|A_{u_0}(u)\|_{X(T)} \leq C \|u_0\|_{H^{a,b}} + C_1 \|\; \|u\|_{L^p_TW^{\kappa,q}(M)}^r \|u\|_{L^\infty_TH^{a,b}}\|1\|_{L^{p/(p-r)}_T}.
$$
Using that $u \in B$, we get
$$
\|A_{u_0}(u)\|_{X(T)} \leq C R + C_1 (2CR)^{r+1}|2T|^{(p-r)/p}.
$$
Taking $T \leq \frac1{2^{(2p-r)/(p-r)} C_1^{p/(p_r)}(2CR)^{rp/(p-r)}} = \frac1{C_2 R^{pr/(p-r)}}$, we get that
$$
\|A_{u_0}(u)\|_{X(T)} \leq 2CR
$$
hence $B$ is stable under $A_{u_0}$. 

We proceed in the same way for the contraction and get
$$
\|A_{u_0}(u) - A_{u_0}(v)\|_{X(T)}  \leq C_3 R^{r} T^{(p-r)/p}\|u-v\|_{X(T)} .
$$
Hence for $T \leq \Big(\frac1{2C_3R^r}\Big)^{p/(p-r)} = \frac1{C_4R^{rp/(p-r)}}$, we get
$$
\|A_{u_0}(u) - A_{u_0}(v)\|_{X(T)}  \leq \frac12 \|u-v\|_{X(T)} 
$$
which make $A_{u_0}$ contracting which implies Theorem \ref{teo3}.
\end{proof}

\subsection{Proof of Theorem \ref{corsolmod}}

The crucial remark here is that the first partial wave subspaces are left invariant by the action of the nonlinear term $|\langle \beta u,u\rangle|^p u$ for every $p$. More precisely, it is possible to prove the following
\begin{lemma}\label{1/2}
Let $j=1/2$ and let $(m_{1/2}, k_{1/2})$ be one of the couples
$$
(-1/2,-1), (-1/2,1), (1/2,-1), (1/2,1).
$$
Then the partial wave subspaces $C^\infty_0((0,\infty),dr)\otimes  \mathcal{H}_{m_{1/2},k_{1/2}}$ are invariant for the nonlinear terms $F(u)=\langle u,u\rangle u$ and $\langle \beta u,u\rangle u$, i.e. \begin{equation}\label{stab}
u\in C^\infty_0((0,\infty),dr)\otimes  \mathcal{H}_{m_{1/2},k_{1/2}}\Rightarrow F(u)\in C^\infty_0((0,\infty),dr)\otimes  \mathcal{H}_{m_{1/2},k_{1/2}}.
\end{equation}
\end{lemma}

\begin{proof}
The proof of this result is already contained in \cite{cac1}, but we report it here for the sake of completeness. We write down for $j=1/2$ the functions $\Phi^+$, $\Phi^-$ that span $\mathcal H_{j,m_j,k_j}$ in the four cases: they turn out to be
\begin{equation}\label{phi1}
\Phi^+_{-1/2,-1}=
\left(\begin{array}
{cc}
0\\ \displaystyle\frac{i}{2\sqrt\pi}\\
0\\0
\end{array}\right)\qquad
\Phi^-_{-1/2,-1}=
\left(\begin{array}
{cc}
0\\0\\
\displaystyle\frac{1}{2\sqrt\pi}e^{i\phi}\sin\theta\\
-\displaystyle\frac{1}{2\sqrt\pi}\cos\theta
\end{array}\right),
\end{equation}

\begin{equation}\label{phi2}
\Phi^+_{-1/2,1}=
\left(\begin{array}
{cc}
\displaystyle\frac{i}{2\sqrt\pi}e^{i\phi}\sin\theta\\
\displaystyle
-\frac{i}{2\sqrt\pi}\cos\theta\\
0\\0
\end{array}\right)\qquad
\Phi^-_{-1/2,1}=
\left(\begin{array}
{cc}
0\\0\\
0\\\displaystyle\frac{1}{2\sqrt\pi}
\end{array}\right),
\end{equation}

\begin{equation}\label{phi3}
\Phi^+_{1/2,-1}=
\left(\begin{array}
{cc}
\displaystyle\frac{i}{2\sqrt\pi}\\
0 \\
0\\0
\end{array}\right)\qquad
\Phi^-_{1/2,-1}=
\left(\begin{array}
{cc}
0\\0\\
\displaystyle\frac{1}{2\sqrt\pi}\cos\theta\\
\displaystyle\frac{1}{2\sqrt\pi}e^{i\phi}\sin\theta
\end{array}\right),
\end{equation}

\begin{equation}\label{phi4}
\Phi^+_{1/2,1}=
\left(\begin{array}
{cc}
\displaystyle\frac{i}{2\sqrt\pi}\cos\theta\\
\displaystyle
\frac{i}{2\sqrt\pi}e^{i\phi}\sin\theta\\
0\\0
\end{array}\right)\qquad
\Phi^-_{1/2,1}=
\left(\begin{array}
{cc}
0\\0\\
\displaystyle\frac{1}{2\sqrt\pi}\\0
\end{array}\right).
\end{equation}
We prove our statement for the couple $(1/2,1)$, i.e. for functions of the form (\ref{phi4}), being the proof for the other cases completely analogous. The generic function $u\in L^2((0,\infty),dr)\otimes  \mathcal{H}_{1/2,1}$ can be written as
$$
u(r,\theta,\phi)=u^+(r)\Phi^+_{1/2,1}(\theta,\phi)+u^-(r)\Phi^-_{1/2,1}(\theta,\phi)
$$
for some radial functions $u^+$, $u^-$. So $u$ takes the vectorial form
\begin{equation}\label{generic}
u=\left(\begin{array}
{cc}
\displaystyle u^+(r)\frac{i}{2\sqrt\pi}\cos\theta\\
\displaystyle
u^+(r)\frac{i}{2\sqrt\pi}e^{i\phi}\sin\theta\\
\displaystyle\frac{u^-(r)}{2\sqrt\pi}\\0
\end{array}\right).
\end{equation}
Thus the Hermitian product $\langle u,u\rangle$ yields
$$
\langle u,u\rangle=\displaystyle \frac{1}{4\pi}\cos^2\theta\: |u^+(r)|^2
+\frac{1}{4\pi}\sin^2\theta\: |u^+(r)|^2+\frac{1}{4\pi}|u^-(r)|^2=
$$
$$
=\frac{1}{4\pi}\left(|u^+(r)|^2+|u^-(r)|^2\right)
$$
that has no angular components. This proves that if $u\in C^\infty_0((0,\infty),dr)\otimes 
\mathcal{H}_{m_{1/2},k_{1/2}}$ then $\langle u,u\rangle u\in C^\infty_0((0,\infty),dr)\otimes  \mathcal{H}_{m_{1/2},k_{1/2}}$. Minor modifications yield the same result also for the nonlinear term $\langle \beta u,u\rangle u$. We have remarked indeed that the operator $\beta$ acts on the partial wave subspaces in a very simple way with respect to the basis $\{\Phi^+,\Phi^-\}$: if in fact we associate to the function $u$ its coordinates $(u^+(r),u^-(r))$ with respect to such a basis we have $\beta u=(u^+(r),-u^-(r))$, so that
$$
\langle \beta u,u\rangle=\displaystyle \frac{1}{4\pi}\cos^2\theta\: |u^+(r)|^2
+\frac{1}{4\pi}\sin^2\theta\: |u^+(r)|^2-\frac{1}{4\pi}|u^-(r)|^2=
$$
$$
=\frac{1}{4\pi}\left(|u^+(r)|^2-|u^-(r)|^2\right)
$$
that again has no angular components, and this shows that if $u\in C^\infty_0((0,\infty),dr)\otimes 
\mathcal{H}_{m_{1/2},k_{1/2}}$ then $\langle \beta u,u\rangle u\in C^\infty_0((0,\infty),dr)\otimes  \mathcal{H}_{m_{1/2},k_{1/2}}$.
\end{proof}

\begin{proof}[Proof of Theorem \ref{corsolmod}] The theorem is a direct consequence of Theorem \ref{teo3} and Lemma \ref{1/2}. We omit the details.

\end{proof}

\section{Acknowledgements} The second author is supported by ANR-18-CE40-0028 project ESSED.

\end{document}